\documentclass[a4paper]{amsart}
\usepackage{a4wide}
\usepackage[english]{babel}
\usepackage[utf8x]{inputenc}
\usepackage{amsmath,amssymb,amsthm}
\usepackage[table]{xcolor}
\usepackage{graphicx,enumerate,url}
\usepackage[colorinlistoftodos]{todonotes}
\usepackage{footmisc}

\usepackage{tikz}
\usetikzlibrary{calc,through,backgrounds,shapes,matrix,decorations.markings,decorations.pathmorphing,arrows.meta}

\usepackage{array}
\newcolumntype{C}[1]{>{\centering\arraybackslash$}p{#1}<{$}}
\usepackage{multirow}
\usepackage{multicol}
\usepackage{hhline}

\usepackage{hyperref}
\usepackage{enumitem}
\setenumerate{label=\normalfont{(\alph*)}, itemsep=5pt, topsep=5pt}

\usepackage{url, hypcap}
\hypersetup{colorlinks=true, citecolor=darkblue, linkcolor=darkblue}

\numberwithin{equation}{section}

\usepackage{cleveref}
\Crefname{algocf}{Algorithm}{Algorithms}

\usepackage[colorinlistoftodos]{todonotes}
\usepackage[boxed,vlined,algosection]{algorithm2e}
\makeatletter
\renewcommand{\@algocf@capt@boxed}{above}
\makeatother

\newtheorem{theorem}{Theorem}[section]
\newtheorem{proposition}[theorem]{Proposition}
\newtheorem{lemma}[theorem]{Lemma}
\newtheorem{corollary}[theorem]{Corollary}
\newtheorem{definition}[theorem]{Definition}

\theoremstyle{definition}
\newtheorem{example}[theorem]{Example}
\newtheorem{remark}[theorem]{Remark}

\newtheorem{algo}[theorem]{Algorithm}

\DeclareMathOperator{\conv}{conv}

\DeclareMathOperator{\cone}{cone}

\newcommand{\RR}{\mathbb{R}}

\newcommand{\RRpos}{\RR_+}

\newcommand{\pcone}[2]{\operatorname{cone}^{(#2)}(#1)}

\newcommand{\set}[2]{\left\{ #1 \;|\; #2 \right\}}
\newcommand{\bigset}[2]{\left\{ #1 \;\big|\; #2 \right\}}

\newcommand{\multiset}[2]{\left\{\!\!\left\{ #1 \;|\; #2 \right\}\!\!\right\}}
\newcommand{\multibracket}[1]{\left\{\!\!\left\{ #1 \right\}\!\!\right\}}

\newcommand{\sref}{\mathcal{S}}
\newcommand{\tref}{\mathcal{R}}
\newcommand{\wo}{{w_\circ}}
\newcommand{\one}{{e}}
\newcommand{\emptyword}{{\varepsilon}}
\newcommand{\csys}{{(W,\sref)}}

\newcommand{\Cplus}[2]{{\mathcal{C}^+(#1,#2)}}
\newcommand{\Cminus}[2]{{\mathcal{C}^-(#1,#2)}}
\newcommand{\Eplus}[2]{{\mathcal{E}^+(#1,#2)}}
\newcommand{\Eminus}[2]{{\mathcal{E}^-(#1,#2)}}
\newcommand{\Phiplus}{\Phi^+}
\newcommand{\Phiminus}{\Phi^-}

\newcommand{\precweak}{\prec_\mathsf{R}}
\newcommand{\leqweak}{\leq_\mathsf{R}}
\newcommand{\inv}{\operatorname{Inv}}
\newcommand{\weakint}[2]{[#1,#2]_{\mathsf{R}}}
\newcommand{\idweak}[1]{{\operatorname{Id}_{\mathsf{R}}(#1)}}

\newcommand{\precbru}{\prec}
\newcommand{\leqbru}{\leq}

\newcommand{\subwordComplex}{\mathcal{SC}}
\newcommand{\greedy}{I_{\operatorname{g}}}
\newcommand{\antigreedy}{I_{\operatorname{ag}}}
\newcommand{\demazure}[1]{{\operatorname{Dem}}(#1)}

\newcommand{\brickVector}[1]{{{\sf b}({#1})}}

\newcommand{\brickPolytope}{\mathcal{B}}

\newcommand{\Root}[2]{{\sfr}(#1,#2)}
\newcommand{\Roots}[1]{{\sfR}(#1)}

\newcommand{\RootsPos}[2]{{\operatorname{Pos}_{#2}(#1)}}

\newcommand{\Weight}[2]{{\sfw}(#1,#2)}

\newcommand{\ie}{\textit{i.e.}}

\newcommand{\Dfn}[1]{\emph{\bfseries #1}}

\newcommand{\sq}[1]{{\rm #1}}
\newcommand{\Q}{\sq{Q}}

\newcommand{\s}{\sq{s}}

\newcommand{\w}{\sq{w}}

\newcommand{\sfr}{{\sf r}}
\newcommand{\sfR}{{\sf R}}
\newcommand{\sfw}{{\sf w}}

\newcommand{\mi}[1]{{\overline{#1}}}
\newcommand{\del}[1]{{{#1}_\Delta}}
\newcommand{\delmin}[1]{{\overline{#1}_\Delta}}

\definecolor{darkblue}{rgb}{0,0,0.7}

\definecolor{lightblue}{rgb}{0.68,0.85,1}

\definecolor{lightgrey}{rgb}{0.9,0.9,0.9}

\definecolor{grey}{rgb}{0.5,0.5,0.5}

\newcommand{\markedbox}[1]{{\bf{\color{red} #1}}}

\newcommand{\normalfan}[1]{{\mathcal{N}(#1)}}
\newcommand{\normalcone}[1]{{\operatorname{C}^{\diamond}(#1)}}
\newcommand{\bignormalcone}[1]{{\operatorname{C}^{\diamond}\big(#1\big)}}

\newcommand{\coxeterfan}[1]{{\mathcal{CF}_{#1}}}

\newcommand{\wordprod}[2]{\Pi{#1}_{#2}} %

\title[Bruhat intervals, subword complexes and brick polyhedra]{Bruhat intervals, subword complexes and brick polyhedra for finite Coxeter groups}

\author[D.~Jahn]{Dennis Jahn}
\address[D.~Jahn]{Fakultät für Mathematik, Ruhr-Universität Bochum, Germany}
\email{dennis.jahn@rub.de}

\author[C.~Stump]{Christian Stump}
\address[C.~Stump]{Fakultät für Mathematik, Ruhr-Universität Bochum, Germany}
\email{christian.stump@rub.de}

\thanks{
  The authors are supported by the DFG Heisenberg grant STU 563/4-1 ``Noncrossing phenomena in Algebra and Geometry''.
}

\begin{document}

\begin{abstract}
  We study the interplay between the discrete geometry of Bruhat poset intervals and subword complexes of finite Coxeter systems.
  We establish connections between the cones generated by cover labels for Bruhat intervals and of root configurations for subword complexes, culminating in the notion of brick polyhedra for general subword complexes.
\end{abstract}

\maketitle
\setcounter{tocdepth}{2}
\tableofcontents

\section{Introduction}
\label{sec:intro}

This paper develops the relationship between Bruhat intervals in finite Coxeter groups and subword complexes by describing the interplay between Bruhat interval cones and root configurations.
Based on this newly developed relationship, properties of brick polytopes for spherical root-independent subword complexes are extended towards general subword complexes by introducing \emph{brick polyhedra}.

\medskip

Bruhat interval cones were introduced and studied by Dyer in the study of positivity properties of Kazhdan-Lusztig and Stanley polynomials~\cite{Dyer-1994}.
Subword complexes are simplicial complexes that have been introduced by Knutson and Miller in the context of Gröbner geometry of Schubert varieties~\cite{KM2004, KM2005}.
Based on the notion of root configurations from~\cite{CLS-14}, this paper reconsiders the combinatorial and discrete-geometric understanding of subword complexes by closely tightening them to Dyer's Bruhat interval cones.

\medskip

Brick polytopes for spherical root-independent subword complexes were introduced and studied by Pilaud and Stump~\cite{Pilaud-Stump-2015}.
These generalized brick polytopes for sorting networks by Pilaud and Santos~\cite{PS-2012}.
As shown by Ceballos, Labbé and Stump in~\cite{CLS-14}, cluster complexes of finite type cluster algebras can be realized as spherical root-independent subword complexes and one main motivation for studying brick polytopes in~\cite{Pilaud-Stump-2015} was to show that generalized associahedra for cluster algebras previously constructed by Chapoton, Fomin and Zelevinsky~\cite{CFZ2002} and by Hohlweg, Lange and Thomas~\cite{HLT2008} can be realized as brick polytopes.

\medskip

Motivated by the conjecture that important properties of brick polytopes for spherical root-independent subword complexes hold for all spherical subword complexes~\cite[Conjecture~7.1]{Pilaud-Stump-2015}, this paper introduces and studies brick polyhedra for general (spherical and non-spherical) subword complexes of finite type.
It turns out that subtle properties of Bruhat interval cones are at the core of their understanding.

\medskip

Subword complexes have a canonical recursive decomposition into smaller subword complexes.
Previous considerations in~\cite{CLS-14,Pilaud-Stump-2015} were mainly developed for spherical subword complexes which are not closed und this recursive structure.
The presented constructions still rely on elementary properties of the two fundamental notions \emph{root} and \emph{weight functions} for subword complexes.
Otherwise, all constructions are based on newly developed properties of Bruhat interval cones that are then applied to general subword complexes.

\medskip

This newly presented approach allows inductive arguments using the recursive structure as done in the proof of \Cref{prop:nonflipablerootincone}.
This is then used to deduce the uniqueness properties of certain facets constructed in \Cref{algo:uniquefacet} that in turn are the central ingredient in the structural understanding of the construction of brick polyhedra.
After recalling standard notions and properties of finite type Coxeter groups, their root systems and their subword complexes in \Cref{sec:background}, we start the discussion of the interplay between Bruhat intervals and subword complexes in \Cref{sec:bruhatcoversandsubwordcomplexes}.
The main results of this interplay are the following, we refer below for proper definitions.
Let $(W,\sref)$ be a finite Coxeter system and let $\subwordComplex(\Q,w)$ be a subword complex associated to a word~$\Q$ in the simple reflections~$\sref$ with Demazure product  $\demazure{\Q}$ and an element~$w \in W$ in the Coxeter group.
\begin{itemize}
  \item \Cref{thm:cone_equality} yields that the Bruhat interval cone of the Bruhat interval $[w,\demazure{\Q}]$ equals the intersection of all cones over root configurations of $\subwordComplex(\Q,w)$.
  \item \Cref{cor:bruhatcone} and \Cref{lem:extendingcovercones,,cor:weakordercplus} exhibits containment properties of Bruhat interval cones that are fundamental for all further considerations.
  \item \Cref{prop:bruhatcoversandrootfunction} shows that the Bruhat interval cone describes the non-flippable vertices in facets.
  \item \Cref{algo:uniquefacet} provides an algorithm to compute $f$-antigreedy facets for linear functionals~$f$ that are non-negative on the Bruhat interval cone.
  \item Based on this algorithm, \Cref{prop:connectedcomp} proves~\cite[Conjeture~7.1]{Pilaud-Stump-2015} for general subword complexes.
\end{itemize}
\Cref{sec:brickpolytopes} then introduces and studies brick polyhedra of subword complexes.
\begin{itemize}
  \item \Cref{def:brickpolytope} defines the brick polyhedron of a finite type subword complex.
  \item \Cref{cor:vertexpointedcone} shows that the local cone of the brick polyhedron at a brick vector coincides with the cone over the root configuration.
  \item \Cref{thm:brickpolynormalcone} shows how to glue together chambers in the Coxeter fan to obtain the normal fan of the brick polyhedron.
  \item \Cref{thm:brickpolycontainment} shows that all brick polyhedra for a fixed word are contained in each other in a natural way.
\end{itemize}

\section{Background on Coxeter systems and subword complexes}
\label{sec:background}

This section recalls standard notions for finite Coxeter systems.
These are mostly following the notions from~\cite{CLS-14,Pilaud-Stump-2015}.
We also refer to~\cite{BB-2005} for further background material.
Ongoing examples are collected in \Cref{sec:examples1}.

\bigskip

Let $\csys$ be a finite type Coxeter system of rank~$n = |\sref|$ acting on a Euclidean vector space~$V \cong \RR^n$ with inner form $\langle \cdot, \cdot \rangle$.
Let $\Delta \subseteq \Phiplus \subseteq \Phi \subseteq V$ be a root system for $\csys$ with simple roots $\Delta = \set{ \alpha_s }{ s \in \sref }$, positive roots~$\Phiplus$, and negative roots $\Phiminus = -\Phiplus$.
For $\beta\in \Phi$ we write $|\beta| \in \Phiplus$ for the positive root in $\{ \pm\beta\}$.
The reflections in~$W$ are $\tref = \set{s_\beta}{\beta \in \Phiplus}$ where $s_\beta$ denotes the reflection sending~$\beta$ to its negative while fixing pointwise its orthogonal complement~$\beta^\perp = \set{v \in V}{\langle \beta, v \rangle = 0}$.

The corresponding Cartan matrix $(a_{st})_{s,t \in \sref}$ is given by $s(\alpha_t) = \alpha_t - a_{st} \alpha_s$.
The fundamental weights $\nabla = \set{ \omega_s }{ s \in \sref } \subseteq V$ are then $\alpha_s = \sum_{t \in \sref} a_{ts} \omega_t$ and~$W$ acts on the fundamental weights by $s(\omega_t) = \omega_t - \delta_{s=t} \alpha_s$ for $s,t \in \sref$.

\subsection{Reduced words and weak order}

The \Dfn{length}~$\ell(w)$ of an element~$w \in W$ is the smallest length of a word~$\s_1 \dots \s_{\ell(w)}$ such that its product is an expression $w = s_1\cdots s_{\ell(w)}$ for~$w$.
Words (and also expressions) of smallest length are called \Dfn{reduced}.
Here and below, we mildly distinguish between words in $\sref$ and the corresponding expressions for elements in~$W$ by writing $\s_1 \dots \s_\ell \in \sref^*$ for the word and $s_1\cdots s_\ell \in W$ for the expression.
It is well-known that for $w \in W, s \in \sref$ we have
\begin{equation}
  \ell(ws) =  \begin{cases}
                \ell(w) + 1 &\text{ if } w(\alpha_s) \in \Phiplus \\
                \ell(w) - 1 &\text{ if } w(\alpha_s) \in \Phiminus
              \end{cases}\ . \label{eq:weaklength}
\end{equation}
This in particular means that a word $\s_1 \dots \s_\ell$ is reduced if and only if for all $1 < i \leq \ell$ we have $s_1\cdots s_{i-1}(\alpha_{s_i}) \in \Phiplus$.
The \Dfn{(right) weak order} $\leqweak$ on~$W$ is the partial order defined by the cover relations $w \precweak ws$ for $w \in W, s \in \sref$ with $\ell(ws) = \ell(w)+1$.
The \Dfn{inversion set} of an element $w \in W$ with reduced expression $w = s_1 \cdots s_\ell$ is given by
\begin{equation}
  \inv(w) = \bigset{s_1\cdots s_{k-1}(\alpha_{s_k})}{1 \leq k \leq \ell}. \label{eq:inversionset}
\end{equation}
This is independent of the chosen reduced word and it is not hard to see that $\inv(w) = \Phiplus \, \cap \, w(\Phiminus)$.
In particular, $\ell(w) = |\inv(w)|$ and $x \leqweak y \Leftrightarrow \inv(x) \subseteq \inv(y)$.
The \Dfn{longest element} $\wo \in W$ is the unique element of longest length $|\Phiplus|$ which is equivalently described as the unique element with $\inv(\wo) = \Phiplus$.
For any $w \in W$, we have a decomposition
\begin{equation}
\label{eq:phiplusdecomposition}
  \Phiplus = \inv(w) \sqcup \inv(w\cdot\wo),
\end{equation}
and a subset $X \subseteq \Phiplus$ is the inversion set of an element if and only if~$X$ and $\Phiplus \setminus X$ are separated by a hyperplane, \ie, there is a linear functional~$f : V \rightarrow \RR$ such that for all $\beta \in \Phiplus$, we have
\begin{equation}
\label{eq:separatedsets}
  f(\beta) > 0 \Longleftrightarrow \beta \in X.
\end{equation}

We record the following well-known lemma.
\begin{lemma}[{\cite[Lemma 1.4.4]{BB-2005}}]
  \label{lem:techlem}
  Let $w \in W$ with reduced expression $w = s_1 \cdots s_\ell$ and let $s_\beta \in \tref$.
  Then the following properties are equivalent:
  \begin{enumerate}
    \item $\ell(s_\beta w) < \ell(w)$,
    \item $s_\beta w = s_1 \cdots \widehat s_k \cdots s_\ell$ for some index~$k$,
    \item\label{it:betaisinversion} $\beta = s_1 \cdots s_{k-1}(\alpha_{s_k})$ for some index~$k$.
  \end{enumerate}
  Moreover, the index~$k$ is unique in both cases.
\end{lemma}

Since~\ref{it:betaisinversion} is the defining property for $\beta$ to be an inversion of~$w$, we obtain from this lemma that
\begin{equation}
\label{eq:inversionandbruhat}
  \inv(w) = \bigset{ \beta \in \Phiplus}{\ell(s_\beta w) < \ell(w) }.
\end{equation}

\subsection{Bruhat order}

The \Dfn{Bruhat order} $\leqbru$ on~$W$ is the partial order defined by the cover relations $w \prec rw$ for $w \in W, r \in \tref$ with $\ell(rw) = \ell(w)+1$.
Observe that $w \precweak ws \Rightarrow w \precbru ws$ because $ws = (wsw^{-1})w$ for $wsw^{-1} \in \tref$.
For the positive root $\beta \in \Phiplus$ with $r = s_\beta$, we say that the cover relation $w \prec rw$ is \Dfn{labelled by}~$\beta$.
In particular, the cover $w \prec ws$ is labelled by $w(\alpha_s)$ which is a positive root by~\eqref{eq:weaklength}.
A \Dfn{Bruhat interval} $[x,y]$ is defined for $x \leq y$ in Bruhat order by
\[
  [x,y] = \bigset{ z \in W}{x \leq z \leq y}.
\]
The following properties of the Bruhat order are needed in subsequent sections.

\begin{lemma}[{\cite[Proposition 2.2.7]{BB-2005}}]
\label{lem:liftingproperty}
  Let $x,y \in W$ and $s \in \sref$ such that $x \leq y$, $x \prec sx$ and $sy \prec y$.
  Then $x \leq sy$ and $sx \leq y$.
\end{lemma}

This lemma immediately yields the following known lifting property for Bruhat intervals.

\begin{corollary}
\label{cor:liftingproperty}
  Let $x,y \in W$ and $s \in \sref$ such that $x \leq y$, $x \prec sx$, $sy \prec y$ and $sx \not\leq sy$.
  Then $[x,sy] \cong [sx,y]$ and the isomorphism is given by $w \mapsto sw$.
\end{corollary}

\begin{lemma}[{\cite[Corollary 2.2.8]{BB-2005}}]
  \label{lem:bruhatdiamond}
  Let $w \in W$, $s \in \sref$ and $r \in \tref$ such that $s \neq r$ and $w \prec sw, rw$.
  Then $sw, rw \prec srw$.
\end{lemma}

Putting these properties together, we obtain the following proposition.

\begin{proposition}
  \label{lem:extendeddiamond}
  Let $x,y \in W$, $s \in \sref$ and $r \in \tref$ such that $s \neq r$ and $x \prec sx, rx \leq y$.
  If moreover $srx \not\leq y$ then $y \prec sy$, $srx \leq sy$ and $[rx,y] \cong [srx,sy]$.
\end{proposition}

\begin{proof}
  \Cref{lem:bruhatdiamond} yields $x \prec sx, rx \prec srx$.
  Assuming $sy \prec y$ and applying \Cref{lem:liftingproperty} to $rx \prec srx$ and to $sy \prec y$ would imply $srx \leq y$, a contradiction.
  Hence $y \prec sy$.
  Again by \Cref{lem:liftingproperty} applied to $rx \prec srx$ and to $y \prec sy$ we obtain $srx \leq sy$ and by \Cref{cor:liftingproperty} we get $[rx,y] \cong [srx,sy]$.
\end{proof}

For later reference, we further recollect the following property.

\begin{lemma}[{\cite[Lemma 2.2.10]{BB-2005}}]
  \label{lem:realtechlem}
  Let $x,y \in W$ and $r \in \tref$.
  If both, $x \leq xr$ and $y \leq ry$, then $xy \leq xry$.
\end{lemma}

\subsection{The Demazure product}

For a (not necessarily reduced) word $\Q = \s_1 \dots \s_m$ in $\sref$ and a subset~$X \subseteq \{1,\dots,m\}$, we write $\Q_X$ to be the subword of~$\Q$ of the letters $\s_i$ with $i \in X$.
The \Dfn{Demazure product} $\demazure{\Q} \in W$ may be defined in terms of the weak order recursively by
\[
  \demazure{\Q} =  \begin{cases}
                     \demazure{\Q_{\{1,\dots,m-1\}}} \cdot s_m &\text{ if } \demazure{\Q_{\{1,\dots,m-1\}}}(\alpha_{s_m}) \in \Phiplus \\
                     \demazure{\Q_{\{1,\dots,m-1\}}}           &\text{ if } \demazure{\Q_{\{1,\dots,m-1\}}}(\alpha_{s_m}) \in \Phiminus
                   \end{cases}\ ,
\]
with initial condition $\demazure{\emptyword} = \one \in W$ for the empty word~$\emptyword$.
By~\eqref{eq:weaklength}, these conditions say that the Demazure product of~$\Q$ is obtained by scanning through~$\Q$ from left to right, starting at the identity and going upwards in weak order whenever possible.
One may equivalently define the Demazure product in terms of the Bruhat order as the unique maximal element in Bruhat order among all expressions obtained from the word $\Q = \s_1 \dots \s_m$ by removing letters,
\begin{equation}
  \demazure{\Q} = \max{}_{\leq}\bigset{\wordprod{\Q}{X}}{X \subseteq \{1,\dots,m\}}. \label{eq:demazurestrongdef}
\end{equation}
Here and below, $\wordprod{\Q}{X} \in W$ denotes the product of the simple reflections $s_i$ with $i \in X$ in the given order.

\subsection{Subword complexes for Coxeter systems}
\label{sec:subwordcomplexes}

For $w \in W$ and $\Q = \s_1 \dots \s_m$, the \Dfn{subword complex} $\subwordComplex(\Q,w)$ is the simplicial complex of sets of (positions of) letters in~$\Q$ whose complements contain reduced words for~$w$.
Knutson and Miller introduced subword complexes in the context of Gröbner geometry of Schubert varieties~\cite{KM2004, KM2005}.
By~\eqref{eq:demazurestrongdef}, $\subwordComplex(\Q,w)$ is non-empty if and only if $w \leq \demazure{\Q}$.
Facets of $\subwordComplex(\Q,w)$ are subwords of~$\Q$ whose complements are reduced words for~$w$.
As shown in \cite[Theorem 3.7, Corollary 3.8]{KM2004} the non-empty complex $\subwordComplex(\Q,w)$ is a topological sphere if and only if $w = \demazure{\Q}$ and a topological ball otherwise.
In these cases we call $\subwordComplex(\Q,w)$ \Dfn{spherical} and \Dfn{non-spherical}, respectively.

\medskip

For two facets $I \neq J$ of $\subwordComplex(\Q,w)$ we say~$I$ and~$J$ are \Dfn{adjacent} if there are $i \in I$ and $j \in J$ such that $I \setminus \{i\} = J \setminus \{j\}$.
We call the transition from~$I$ to~$J$ the \Dfn{flip} of~$i$ in~$I$ and furthermore say the flip is \Dfn{increasing} if $i < j$ and \Dfn{decreasing} otherwise.
The \Dfn{root function} $\Root{I}{\cdot} : [m] \rightarrow \Phi = W(\Delta) \subseteq V$ is defined by
\[
  \Root{I}{k} = \wordprod{\Q}{\{1,\dots,k-1\} \setminus I}(\alpha_{s_k}).
\]
As above, $\wordprod{\Q}{X}$ denotes the product of the simple reflections $s_i$, for $i \in X$ in the given order.
Observe here that the root $\Root{I}{k}$ only depends on $I \cap \{1,\dots,k-1\}$ and not on the complete facet~$I$.
The ordered multiset\footnote{\label{fn:orderedsets}We think of facets of subword complexes as being \emph{ordered sets} (sorted lists of indices written in set notation) and we think of root configurations of facets as \emph{ordered multisets} (lists of roots written in multiset notation with ordering inherited from the order of the facet).} $\Roots{I} = \multiset{\Root{I}{i}}{i \in I}$ is called \Dfn{root configuration} of the facet~$I$.

\medskip

The following lemma recalls several properties of flips in subword complexes in terms of root configurations.
These can be found in \cite[Proposition~2]{MR3156785}.

\begin{lemma}
  \label{lem:rootsandflips}
  Let $I$ be a facet of the non-empty subword complex $\subwordComplex(\Q,w)$.
  \begin{enumerate}
    \item The map $\Root{I}{\cdot} : k \mapsto \Root{I}{k}$ is a bijection between the complement of~$I$ and~$\inv(w)$.

    \item\label{eq:rootsbyflips} For $i \in I$, there exists a facet~$J$ and an index $j \in J$ such that $I \setminus \{i\} = J \setminus \{j\}$ if and only if $|\Root{I}{i}| \in \inv(w)$.
    In this case, the index~$j$ is the unique position in the complement of~$I$ for which $\Root{I}{j} = |\Root{I}{i}|$ and the sign is determined by
    \[
      \Root{I}{j} = \begin{cases}
                       \Root{I}{i} &\text{if } i < j, \\
                      -\Root{I}{i} &\text{if } i > j.
                    \end{cases}
    \]

    \item In the situation of~\ref{eq:rootsbyflips}, the map $\Root{J}{\cdot}$ is obtained from $\Root{I}{\cdot}$ by
    \[
      \Root{J}{k} = 
        \begin{cases}
        s_{\Root{I}{i}} \big(\Root{I}{k}\big) & \text{if } \min(i,j) < k \leq \max(i,j), \\
        \Root{I}{k}                   & \text{otherwise. }
        \end{cases}
    \]
  \end{enumerate}
\end{lemma}

Non-empty subword complexes come with two extremal facets.
Considering facets as ordered sets\footref{fn:orderedsets}, the \Dfn{greedy facet} $\greedy$ is the lexicographically first facet and the \Dfn{antigreedy facet} $\antigreedy$ is the lexicographically last facet.
Using \Cref{lem:rootsandflips}, one can describe these greedy facets by their root configurations.
We also refer to \cite[Section~4.2]{MR3156785} for a discussion of greedy facets.

\begin{lemma}
\label{lem:greedyfacet}
  Let $\subwordComplex(\Q,w)$ be a non-empty subword complex.
  Its greedy facet $\greedy$ is the unique facet with $\Roots{\greedy} \subseteq \Phiplus$ and its antigreedy facet $\antigreedy$ is the unique facet with $\Roots{\antigreedy} \subseteq w(\Phiplus)$.
\end{lemma}

\begin{proof}
  As shown in \cite[Proposition~6]{MR3156785}, the greedy facet is the unique facet for which all flips are increasing, and the antigreedy facet is the unique facet for which all flips are decreasing.
  These properties are equivalent to the respective descriptions in terms of root configurations:
  For the greedy facet, this immediately follows from \Cref{lem:rootsandflips}\ref{eq:rootsbyflips}, while it follows for the antigreedy facet from this lemma together with the computation
  \begin{align*}
    \Roots{\antigreedy} &\subseteq \big(\Phiplus \setminus \inv(w)\big) \cup -\inv(w) \\
                        &= \big( \Phiplus \setminus \big(\Phiplus \cap -w(\Phiplus  )\big) \big) \cup \big( \Phiminus \cap w(\Phiplus) \big) \\
                        &= \big( \Phiplus \cap w(\Phiplus) \big) \cup \big( \Phiminus \cap w(\Phiplus) \big) \\
                        &= w(\Phiplus),
  \end{align*}
  where the first inclusion is the describing property that all flips are decreasing and the others are elementary reformulations.
\end{proof}

\subsection{Ongoing examples}
\label{sec:examples1}

Throughout this paper we present examples in Coxeter systems of types $A_2$, $B_2$ and $B_3$.
We generally write~$\sref = \{s_1,\dots,s_n\}$ for concrete generators~$s_i$ and also $\Delta = \{\alpha_1,\dots,\alpha_n\}$ for simple roots with $\alpha_i = \alpha_{s_i}$.
To keep examples compact, we write shorthand $\Q = 123212 = \s_1\s_2\s_3\s_2\s_1\s_2$ for words in~$\sref$ and also for elements $w = 1231 = s_1s_2s_3s_1 \in W$ as reduced words, and we abbreviate $\del{213} = 2\alpha_1+\alpha_2+3\alpha_3$ and $\delmin{213} = -\del{213}$ for vectors written in the basis of simple roots.

\begin{example}[Type $A_2$]
  We have $\sref = \{ s_1,s_2 \}$, the simple and positive roots $\Delta \subset \Phiplus$ given by
  $
    \{ \del{10}, \del{01} \}\subset \{ \del{10},\del{01},\del{11} \}
  $
  and
  \[
    s_1(\del{10}) = \delmin{10}, \quad s_1(\del{01}) = \del{11},\quad
    s_2(\del{10}) = \del{11}, \quad s_2(\del{01}) = \delmin{01}.
  \]
  The fundamental weights are
  $
    \nabla = \{ \omega_1 = \tfrac{1}{3}(\del{21}), \omega_2 = \tfrac{1}{3}(\del{12}) \}.
  $
\end{example}

\begin{example}[Type $B_2$]
  We have $\sref = \{ s_1,s_2 \}$, the simple and positive roots $\Delta \subset \Phiplus$ given by
  $
    \{ \del{10}, \del{01} \}\subset \{ \del{10}, \del{01}, \del{11}, \del{12} \}
  $
  and
  \[
    s_1(\del{10}) = \delmin{10}, \quad s_1(\del{01}) = \del{11}, \quad
    s_2(\del{10}) = \del{12}, \quad s_2(\del{01}) = \delmin{01}.
  \]
  The fundamental weights are
  $
  \nabla = \{ \omega_1 = \del{11}, \omega_2 = \tfrac{1}{2}(\del{12}) \}.
  $
  \end{example}
  
\begin{example}[Type $B_3$]
  We have $\sref = \{ s_1,s_2,s_3 \}$, the simple and positive roots $\Delta \subset \Phiplus$ given by
  $
  \{ \del{100}, \del{010}, \del{001} \}\subset \{ \del{100},\del{010},\del{001},\del{110},\del{011},\del{111},\del{012},\del{112},\del{122} \}
  $
  and
  \begin{align*}
    s_1(\del{100}) = \delmin{100}, \quad s_1(\del{010}) = \del{110}, \quad s_1(\del{001}) = \del{001}, \\
    s_2(\del{100}) = \del{110}, \quad s_2(\del{010}) = \delmin{010}, \quad s_2(\del{001}) = \del{011}, \\
    s_3(\del{100}) = \del{100}, \quad s_3(\del{010}) = \del{012}, \quad s_3(\del{001}) = \delmin{001}, \\
  \end{align*}
  The fundamental weights are
  $
  \nabla = \{ \omega_1 = (\del{111}), \omega_2 = (\del{110}), \omega_3 = \tfrac{1}{2}(\del{112}) \}.
  $
\end{example}

\section{The interplay between Bruhat intervals and subword complexes}
\label{sec:bruhatcoversandsubwordcomplexes}

This section studies the interplay between Bruhat intervals and subword complexes via root configurations.
To this end, we recall the \Dfn{(closed) cone} over a finite set~$X \subseteq V$ to be
\[
  \cone(X) = \textstyle\sum_{v \in X}\RR_+ v = \set{ \textstyle\sum_{v \in X} \lambda_vv}{\lambda_v \geq 0}.
\]
The following notions of Bruhat cones were studied by Dyer in the context of Kazhdan-Lusztig polynomials~\cite{Dyer-1994}.
For a Bruhat interval $[x,y]$, the \Dfn{upper Bruhat cone} $\Cplus{x}{y}$ and the \Dfn{lower Bruhat cone} $\Cminus{x}{y}$ are defined by $\Cplus{x}{y} = \cone\Eplus{x}{y}$ and $\Cminus{x}{y} = \cone\Eminus{x}{y}$ for
\[
  \Eplus{x}{y}  = \bigset{\beta \in \Phiplus}{x \prec s_\beta x \leq y}, \quad
  \Eminus{x}{y} = \bigset{\beta \in \Phiplus}{x \leq s_\beta y \prec y}.
\]
This is, the upper Bruhat cone is the cone spanned by the labels of the atoms in the Bruhat interval $[x,y]$ and the lower Bruhat cone is the cone spanned by the labels of the coatoms.
Since we mainly discuss upper Bruhat cones, we simply call these \Dfn{Bruhat cones}.

\medskip

The goal of this section is to develop tools connecting Bruhat cones and root configurations, culminating in the following theorem.

\begin{theorem}
\label{thm:cone_equality}
  Let~$\Q$ be a word in~$\sref$ and let $w \in W$ with $w \leq \demazure{\Q}$.
  Then
  \[
    \Cplus{w}{\demazure{\Q}} = \bigcap_I \cone\Roots{I}
  \]
  where the intersection is taken over all facets~$I$ of the subword complex $\subwordComplex(\Q,w)$.
\end{theorem}

\subsection{Properties of Bruhat cones}
\label{sec:conesinbruhatintervals}

We first recall several properties of Bruhat cones from \cite{Dyer-1994} and then exhibit further recursive properties of these.

\begin{lemma}{\cite[Proposition 1.4, Proposition 3.6]{Dyer-1994}}
  \label{lem:dyer}
  Let $[x,y]$ be a Bruhat interval and let $w \in W$.
  Then
  \begin{enumerate}
    \item\label{it:dyer-a} $\Cminus{e}{w} \cap \Cplus{w}{\wo} = \{0\}$ and $\Phiplus \subseteq \Cminus{\one}{w} \cup \Cplus{w}{\wo}$.
    
    \item\label{it:dyer-c} $\Eplus{x}{y}$ are the rays of $\Cplus{x}{y}$ and $\Eminus{x}{y}$ are the rays of $\Cminus{x}{y}$.
    
    \item\label{it:dyer-d} For $\beta \in \Eplus{x}{y}$ we have $\Cplus{x}{y} \subseteq \Cplus{s_\beta x}{y} + \RR (\beta)$, and similar for $\Cminus{x}{y}$.
    
    \item\label{it:dyer-e} For $\beta \in \Phiplus$ with $x \leq s_\beta x \leq y$ we have $\beta \in \Cplus{x}{y}$, and similar for $\Cminus{x}{y}$.
  \end{enumerate}
\end{lemma}

This lemma has the following consequence, which is interesting in itself while not used in our further considerations.
It can also be directly deduced from~\eqref{eq:separatedsets} and~\eqref{eq:inversionandbruhat}.

\begin{corollary}
  Let $w \in W$.
  Then $\Eminus{\one}{w}$ are the extremal rays of $\cone\inv(w)$ and $\Eplus{w}{\wo} = \Eminus{\one}{w\cdot\wo}$ are the extremal rays of $\cone\inv(w\cdot\wo)$.
  In particular,
  \[
  \Cminus{\one}{w} \cap \Phi = \inv(w), \quad \Cplus{w}{\wo} \cap \Phi = \inv(w\cdot\wo).
  \]
\end{corollary}

\begin{proof}
  It is immediate from~\eqref{eq:inversionandbruhat} that $\Eminus{\one}{w} \subseteq \inv(w)$, and thus $\Eplus{w}{\wo} = \Eminus{\one}{w\cdot\wo} \subseteq \inv(w\cdot\wo)$.
  The statements then follow with~\eqref{eq:separatedsets} and~\eqref{eq:phiplusdecomposition} from \Cref{lem:dyer}\ref{it:dyer-c} and~\ref{it:dyer-a}.
\end{proof}

We provide the following additional properties needed in subsequent considerations.

\begin{theorem}
  \label{cor:bruhatcone}
  Let $x,y \in W$ and $s \in \sref$  such that $x \leq y \prec sy$.
  Then
  \[
  \Cplus{x}{sy} \subseteq \Cplus{x}{y} + \RRpos(\alpha_s).
  \]
\end{theorem}

\begin{corollary}
\label{lem:extendingcovercones}
  Let $x,y \in W$ and $s \in \sref$ such that $sx \prec x \leq y$.
  Let furthermore $\tau \in \{ sy, y \}$ be the Bruhat smaller element.
  Then
  \[
    s\big( \Cplus{x}{y} \big) \subseteq \Cplus{sx}{\tau}.
  \]
\end{corollary}

We also need this corollary for the weak order, so we also provide that version individually.

\begin{corollary}
\label{cor:weakordercplus}
  Let $x,y \in W$ and $s \in \sref$ such that $xs \prec x \leq y$.
  Let furthermore $\tau \in \{ ys, y \}$ be the Bruhat smaller element.
  Then
  \[
    \Cplus{x}{y} \subseteq \Cplus{xs}{\tau}.
  \]
\end{corollary}

The first step in the proof is to strengthen \Cref{lem:dyer}\ref{it:dyer-d} for simple roots in $\Eplus{x}{y}$ as follows.

\begin{proposition}
  \label{prop:bruhatcone}
  Let $x,y \in W$, $s\in \sref$ and $r = s_\beta \in \tref$ such that $s \neq r$ and $x \prec sx, rx \leq y$.
  Then
  \[
    s(\beta) \in \Cplus{sx}{y} + \RRpos(\alpha_s).
  \]
\end{proposition}

\begin{proof}
  By \Cref{lem:bruhatdiamond} we have $x \prec sx, rx \prec srx$ with $srx = s_{s(\beta)} sx$.

  If $srx \leq y$ then
  $
    s(\beta) \in \Cplus{sx}{y} \subseteq \Cplus{sx}{y} + \RRpos(\alpha_s)
  $
  by definition, so we have to consider the situation $srx \not\leq y$.
  In this case $s( \beta ) \not\in \Cplus{sx}{y}$, and the situation with their cover labels is
  
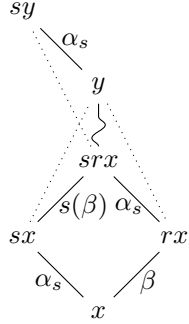
\begin{center}
  \begin{tikzpicture}[scale=0.5]
  \node (x) at (0,0) {$x$};
  \node (sx) at (-2,2) {$sx$};
  \node (rx) at (2,2) {$rx$};
  \node (srx) at (0,4) {$srx$};
  \node (y) at (0,6) {$y$};
  \node (sy) at (-2,8) {$sy$};
  
  \draw (x) -- (sx) -- (srx) -- (rx) -- (x);
  \draw[dotted] (sx) -- (y) -- (rx);
  \draw[dotted] (srx) -- (sy);
  \draw (y) -- (sy);
  \draw[decorate, decoration={snake}] (srx) -- (y);
  
  \node (a1) at (-1.3,0.8) {$\alpha_s$};
  \node (a2) at (0.8,2.8) {$\alpha_s$};
  \node (a3) at (-0.6,7.2) {$\alpha_s$};
  \node (b) at (1.3,0.8) {$\beta$};
  \node (sb) at (-0.4,2.8) {$s(\beta)$};
  \end{tikzpicture}
  \end{center}
  
  We may thus apply \Cref{lem:extendeddiamond} and obtain $y \prec sy$, $srx \leq sy$ and $[rx,y] \cong [srx,sy]$ given by $w \mapsto sw$, and get
  \begin{equation}
  \label{eq:stransform}
    \Cplus{rx}{y} = s \big(\Cplus{srx}{sy}\big). \tag{$\ast$}
  \end{equation}
  By \Cref{lem:dyer}\ref{it:dyer-a} we have $\beta \not\in \Cplus{rx}{sy}$, and by \Cref{lem:dyer}\ref{it:dyer-c} we have $\alpha_s \in \Cplus{rx}{sy}$ is a ray, implying in particular $\alpha_s \not\in \Cplus{rx}{y}$.
  Because $\alpha_s \in \Delta$ is simple and because $\Cplus{rx}{y} \subseteq \Cplus{rx}{sy}$ we obtain
  \[
    \Cplus{rx}{y} \cap P \subseteq \RRpos\{ \alpha_s, \beta \},
  \]
  where the situation is considered inside the two-dimensional plane $P = \langle \alpha_s, \beta \rangle_\RR$.
  This situation is sketched as
  
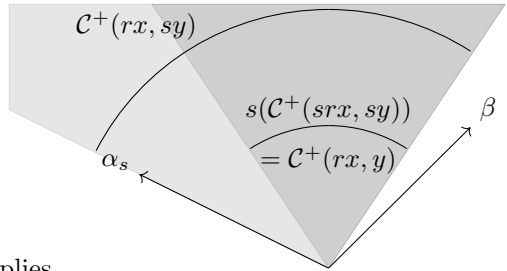
\begin{center}
    \begin{tikzpicture}[scale=0.7]
      \draw[fill=black, black, opacity=0.1] (0,0) -- (-6,3) -- (-6,5) -- (3.32,5) -- (0,0);
      \draw[fill=black, black, opacity=0.1] (0,0) -- (-3.32,5) -- (3.32,5) -- (0,0);

      \draw (0,2.7) arc(90:57:2.7);
      \draw (0,2.7) arc(90:123:2.7);
      
      \draw (0,4.9) arc(90:57:4.9);
      \draw (0,4.9) arc(90:153:4.9);

      \node (a) at (-4,2) {$\alpha_s$};
      \node (b) at (3,3) {$\beta$};
      \draw[->] (0,0) -- (a);
      \draw[->] (0,0) -- (b);

      \node (C1) at (-3.6,4.6) {$\Cplus{rx}{sy}$};
      \node (C2) at (0,3) {$s(\Cplus{srx}{sy})$};
      \node (C3) at (0,2.05) {$=\Cplus{rx}{y}$};
    \end{tikzpicture}
  \vspace*{-15pt}
  \end{center}
  Together with \eqref{eq:stransform}, this implies
  \begin{equation}
  \label{eq:pplane}
    s(\beta) \in \big(\Cplus{srx}{sy} \cap P \big) + \RRpos( \alpha_s ). \tag{$\ast\ast$}
  \end{equation}
  By the assumption $x \prec rx \leq y$, we know $\beta \in \Cplus{x}{y}$.
  Together with \Cref{lem:dyer}\ref{it:dyer-d}, this implies
  \begin{equation}
  \label{eq:nontrivialpoint}
    \beta \in \Cplus{x}{y} \subseteq \Cplus{sx}{y} + \RR (\alpha_s). \tag{$\ast\ast$$\ast$}
  \end{equation}
  The cone $\Cplus{sx}{y}$ therefore intersects the plane~$P$ non-trivially, $\Cplus{sx}{y} \cap P \neq \{0\}$.
  We have $\Cplus{sx}{y} \subseteq \Cplus{sx}{sy}$ by definition and we have $\Cplus{sx}{sy} \subseteq \Cplus{srx}{sy} + \RR ( s(\beta) )$ by \Cref{lem:dyer}\ref{it:dyer-d}.
  So we obtain $\Cplus{sx}{y} \subseteq \Cplus{srx}{sy} + \RR ( s(\beta) )$, sketched as
  
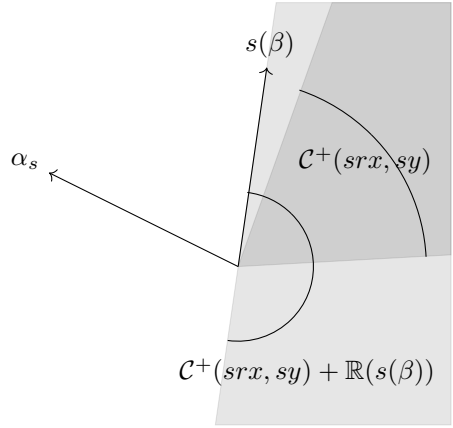
\begin{center}
    \begin{tikzpicture}[scale=0.7]
      \draw[fill=black, black, opacity=0.1] (0,0) -- (0.714,5) -- (4,5) -- (4,-3) -- (-0.429,-3) -- (0,0);
      \draw[fill=black, black, opacity=0.1] (0,0) -- (4,0.222) -- (4,5) -- (1.765,5)  -- (0,0);

      \draw (2.5,2.5) arc(45:3:3.535);
      \draw (2.5,2.5) arc(45:71:3.535);

      \draw (1,1) arc(45:-98:1.414);
      \draw (1,1) arc(45:83:1.414);

      \node (a) at (-4,2) {$\alpha_s$};
      \node (sb) at (0.6,4.2) {$s(\beta)$};
      \draw[->] (0,0) -- (a);
      \draw[->] (0,0) -- (sb);

      \node (C1) at (2.4,2) {$\Cplus{srx}{sy}$};
      \node (C2) at (1.3,-2) {$\Cplus{srx}{sy} + \RR(s(\beta))$};
    \end{tikzpicture}
  \end{center}
  Since $\Cplus{sx}{y}$ is spanned by positive roots, \eqref{eq:pplane} and~\eqref{eq:nontrivialpoint} finally yield
  \[
    s(\beta) \in \big(\Cplus{sx}{y} \cap P \big) + \RRpos(\alpha_s) \subseteq \Cplus{sx}{y} + \RRpos(\alpha_s). \qedhere
  \]
\end{proof}

The following example details the situation in the proof of \Cref{prop:bruhatcone}.

\begin{example}[Type $B_3$]
  Let $x = 13231$ and $r = s_{\del{122}}$.
  Then
  \[
    s_2 x = 213231,\quad rx = 132312,\quad s_2rx = 2132312
  \]
  and we have $x \prec s_2x,rx \prec s_2rx$ in accordance with \Cref{lem:bruhatdiamond}.
  Let $y = 32312321$ so that $s_2x, rx \leq y$ while $s_2rx \not\leq y$.
  We draw the interval $[x, s_2y]$ where crucial cover relations and elements are drawn in black while others are drawn in light grey for better readability.
  
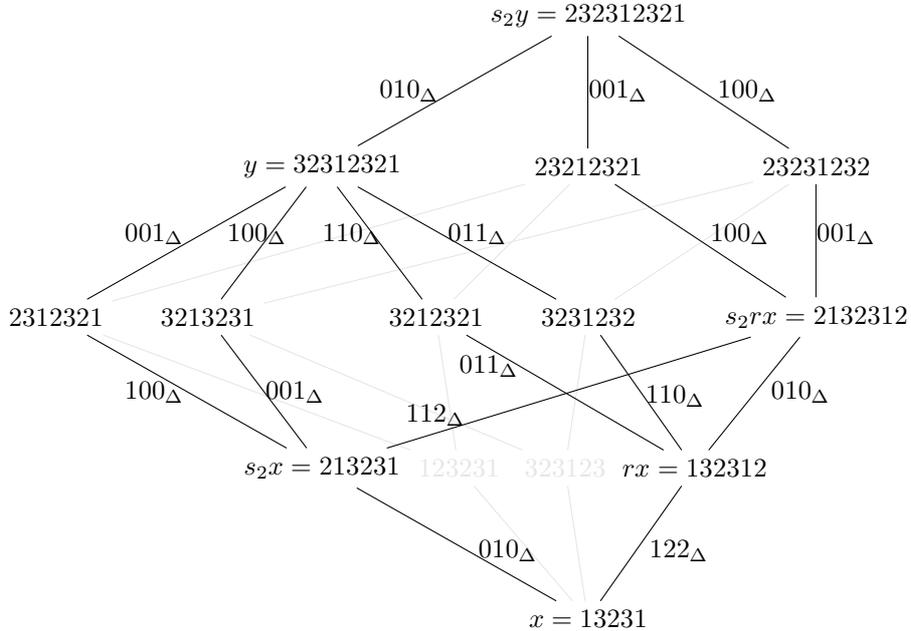
\[
  \begin{tikzpicture}[scale=1.0]
  \node (x)   at (3,0)   {$x = 13231$};
  \node (sx)  at (-0.5,2) {$s_2 x = 213231$};
  \node (tx)  at (4.4,2) {$rx = 132312$};
  \node (stx) at (6,4)   {$s_2 rx = 2132312$};
  \node (c1)  at (-2,4)   {$3213231$};
  \node (c2)  at (-4,4)   {$2312321$};
  
  \node (c3)  at (1,4)   {$3212321$};
  \node (c4)  at (3,4)   {$3231232$};
  
  \node (y)   at (-0.5,6)   {$y = 32312321$};
  \node (d1)  at (3,6)   {$23212321$};
  \node (d2)  at (6,6)   {$23231232$};
  \node (sy)  at (3,8)   {$s_2 y = 232312321$};
  
  \node[lightgrey] (gr1) at (1.3,2) {$123231$};
  \node[lightgrey] (gr2) at (2.7,2) {$323123$};
  
  \draw[lightgrey] (x) -- (gr1);
  \draw[lightgrey] (x) -- (gr2);
  \draw[lightgrey] (gr1) -- (c2);
  \draw[lightgrey] (gr1) -- (c3);
  \draw[lightgrey] (gr2) -- (c1);
  \draw[lightgrey] (gr2) -- (c4);

  \draw[lightgrey] (c3) -- (d1);
  \draw[lightgrey] (c4) -- (d2);
  \draw[lightgrey] (c1) -- (d2);
  \draw[lightgrey] (c2) -- (d1);
  
  \draw (x)   -- (sx) -- (stx) -- (tx) -- (x);
  \draw (sx)  -- (c1) -- (y)   -- (c2) -- (sx);
  \draw (stx) -- (d1) -- (sy)  -- (d2) -- (stx);
  \draw (y)   -- (sy);
  \draw (tx) -- (c3) -- (y) -- (c4) -- (tx);
  
  \node (b1) at (1.95,0.9) {$\del{010}$};
  \node (b2) at (5.8,3)   {$\del{010}$};
  \node (b3) at (0.65,7)     {$\del{010}$};
  
  \node (a1) at (-2.7,3)   {$\del{100}$};
  \node (a2) at (5.1,7)     {$\del{100}$};
  \node (a3) at (5,5.1)   {$\del{100}$};
  \node (a4) at (-1.35,5.1) {$\del{100}$};
  
  \node (g1) at (-0.85,3)    {$\del{001}$};
  \node (g2) at (6.4,5.1)  {$\del{001}$};
  \node (g3) at (3.4,7)    {$\del{001}$};
  \node (g4) at (-2.7,5.1) {$\del{001}$};
  
  \node (o1) at (1,2.7) {$\del{112}$};
  \node (o1) at (4.2,0.9) {$\del{122}$};
  
  \node (ab1) at (-0.1,5.1)    {$\del{110}$};
  \node (ab2) at (4.15,2.95)    {$\del{110}$};
  
  \node (bg1) at (1.55,5.1)    {$\del{011}$}; 
  \node (bg2) at (1.7,3.35)    {$\del{011}$};
  
  \end{tikzpicture}
  \]
  We see $[rx, y] \cong [s_2rx, s_2y]$ realized by $w \mapsto s_2w$ as predicted by \Cref{lem:extendeddiamond} and furthermore
  \[
    \Cplus{rx}{y} = \cone\{\del{110}, \del{011}\} = s_2 \big(\cone\{\del{100}, \del{001}\}\big) = s_2 \big(\Cplus{s_2 rx}{s_2 y}\big).
  \]
  The plane ${P = \langle\del{010}, \del{122}\rangle_\RR}$ contains the sub root system of type $A_2$ with simple roots $\Delta_P = \{ \del{010}, \del{112} \}$ and we have $\Phiplus \cap P = \{\del{010}, \del{122}, \del{112}\}$
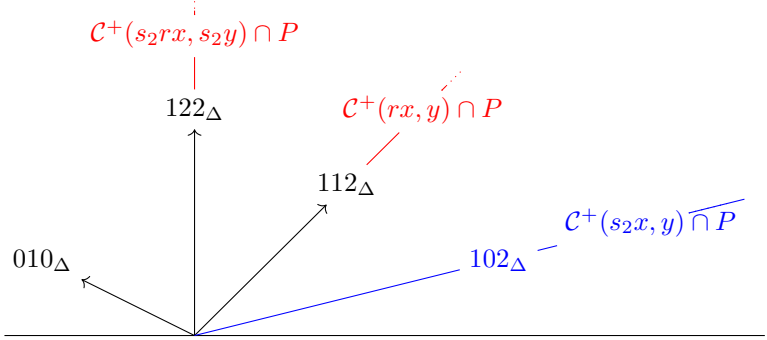
  \[
  \begin{tikzpicture}[scale=0.5]
  \draw     (1,0)  -- (21,0);
  
  \node (j1) at (2,2) {$\del{010}$};
  \node (j2) at (6,6) {$\del{122}$};
  \node (j3) at (10,4) {$\del{112}$};
  \node[blue] (j4) at (14,2) {$\del{102}$};
  
  \draw[->] (6,0) -- (j1);
  \draw[->] (6,0) -- (j2);
  \draw[->] (6,0) -- (j3);
  
  \node[red] (i1) at (6,8) {$\Cplus{s_2 rx}{s_2 y} \cap P$};
  \draw[red]            (j2) -- (i1);
  \draw[red]            (i1) -- (6,8.5);
  \draw[red, dotted] (6,8.5) -- (6,9);
  
  \node[red] (i2) at (12,6) {$\Cplus{rx}{y} \cap P$};
  \draw[red]           (j3) -- (i2);
  \draw[red]           (i2) -- (12.5,6.5);
  \draw[red, dotted] (12.5,6.5) -- (13,7);
  
  \node[blue] (i3) at (18,3) {$\Cplus{s_2 x}{y} \cap P$};
  \draw[blue] (6,0) -- (j4) -- (i3) -- (19,3.25);
  \draw[blue, dotted] (19,3.25) -- (20,3.5); 
  \end{tikzpicture}
  \]
  and see
  $
    s_2( \del{122} ) = \del{112} \in \Cplus{s_2 x}{y} + \cone\{ \del{010} \}
  $
  as predicted by \Cref{prop:bruhatcone}.
\end{example}

\begin{proof}[Proof of \Cref{cor:bruhatcone}]
  Let $r \in \tref$ with $x \prec rx \leq sy$.
  If $rx \leq y$ or $r = s$ the statement is trivial, so let $rx \not\leq y$ and $r \neq s$.
  
  As $rx \leq sy$ while $rx \not\leq y$, we have $y \prec sy$.
  Hence there is a reduced word for~$sy$ that starts with the letter~$\s$ and from which we may delete letters to obtain a reduced word for~$rx$.
  Observe that we must not delete the initial letter~$\s$ in this process, so we obtain a reduced word for~$rx$ with initial letter~$\s$.
  Now by deleting this initial~$\s$, we get $srx \prec rx$ and $srx \leq y$.

  As $r \neq s$ we have $x \neq srx$ and $x, srx \prec rx$.
  We thus obtain
  \[
    sx \prec x, srx \prec rx \text{ and } x, srx \leq y,
  \]
  and conclude with \Cref{prop:bruhatcone}.
\end{proof}

\begin{proof}[Proof of \Cref{lem:extendingcovercones}]
  If $x = y$ we have $\Cplus{x}{y} = \{0\}$ and the statement is trivial, hence let $x < y$.
  We first prove for any $\beta \in \Eplus{x}{y}$, that 
  \begin{equation}\label{eq:extendingcovercones}
    s(\beta) \in \Cplus{sx}{s_\beta x}. \tag{$\ast$}
  \end{equation}
  We have $\ell(s_\beta x) = \ell(sx) + 2$, hence 
  \[
    \ell(s_{s(\beta)} sx) = \ell( s s_\beta x) = \ell(s_\beta x) \pm 1 = \ell(sx) + 2 \pm 1 > \ell(sx),
  \]
  and therefore by \Cref{lem:dyer}\ref{it:dyer-e}, $s(\beta) \in \Cplus{sx}{s s_\beta x}$.
  If $s s_\beta x \prec s_\beta x$ then
  \[
    s(\beta) \in \Eplus{sx}{s s_\beta x} \subseteq \Eplus{sx}{s_\beta x} \subseteq \Cplus{sx}{s_\beta x}.
  \]
  If otherwise $s_\beta x \prec s s_\beta x$ we apply \Cref{cor:bruhatcone} to obtain
  \[
    \Cplus{sx}{s s_\beta x} \subseteq \Cplus{sx}{s_\beta x} + \RRpos( \alpha_s ).
  \]
  But as $sx \prec x \prec s_\beta x$ we have $\Cplus{sx}{s_\beta x} + \RRpos( \alpha_s ) = \Cplus{sx}{s_\beta x}$ and thus $s(\beta) \in \Cplus{sx}{s_\beta x}$.
  This concludes the proof of~\eqref{eq:extendingcovercones}.
  \medskip

  If $\tau = y$ we clearly have $\Cplus{sx}{s_\beta x} \subseteq \Cplus{sx}{y}$ for every $\beta \in \Eplus{x}{y}$, hence by \eqref{eq:extendingcovercones} we get
  \[
    s( \Cplus{x}{y} ) \subseteq \Cplus{sx}{\tau}.
  \]
  If $\tau = sy$ we consider the following cases.
  
  \textbf{Case 1:} 
  If $x \not\leq sy$ we apply \Cref{cor:liftingproperty} to $sx \prec x$ and $sy \prec y$ and obtain $[sx, sy] \cong [x,y]$, hence $s( \Cplus{x}{y} ) = \Cplus{sx}{\tau}$.
  
  \textbf{Case 2:}
  If $x \leq sy$ but $s_\beta x \not \leq sy$ take a reduced word for~$y$ that starts with~$\s$, say
  $
    \s~ \s_1 \dots \s_m,
  $
  so that $\s_1 \dots \s_m$ is a reduced word for $sy = \tau$.
  A reduced word for~$s_\beta x$ can be obtained by deleting letters different from the initial letter~$\s$.
  Therefore~$s$ is a descent of~$s_\beta x$ and furthermore $s s_\beta x \leq \tau$.
  Thus $ s( \beta ) \in \Cplus{sx}{s s_\beta x} \subseteq \Cplus{sx}{\tau}$.
  
  \textbf{Case 3:}
  If $s_\beta x \leq sy$ we apply \eqref{eq:extendingcovercones} to obtain $\s(\beta) \in \Cplus{sx}{s_\beta x} \subseteq \Cplus{sx}{\tau}$ for every $\beta \in \Eplus{x}{y}$.
\end{proof}

\begin{proof}[Proof of \Cref{cor:weakordercplus}]
  We first show the identity
  \begin{equation}\label{eq:weakordercplus}
    \Cplus{x^{-1}}{y^{-1}} = x^{-1} \big{(} \Cplus{x}{y} \big{)}. \tag{$\ast$}
  \end{equation}
  If $x=y$ the statement is trivial, hence let $x<y$ and take $\beta \in \Eplus{x}{y}$.
  We then have $x \prec s_\beta x\leq y$ or, equivalently,
  \[
    x^{-1} \prec x^{-1} s_\beta = s_{x^{-1}(\beta)} x^{-1} \leq y^{-1}.
  \]
  This implies $x^{-1}(\beta) \in \Eplus{x^{-1}}{y^{-1}}$ and we obtain \eqref{eq:weakordercplus}.
  
  \medskip
  
  Again the main statement is trivial for $x=y$, hence let $x<y$ and take $\beta \in \Eplus{x}{y}$.
  We then have similar as above
  \[
    s x^{-1} \prec x^{-1} \prec s_{x^{-1}(\beta)} x^{-1} \leq y^{-1}.
  \]
  By \Cref{lem:extendingcovercones} we get
  \[
    sx^{-1}( \beta ) \in \Cplus{sx^{-1}}{\tau^{-1}}.
  \]
  Applying \eqref{eq:weakordercplus} we finally conclude with
  \[
    xssx^{-1}( \beta ) = \beta \in \Cplus{xs}{\tau}. \qedhere
  \]
\end{proof}

\subsection{Bruhat cones and non-flipable vertices in subword complexes}
\label{sec:nonflipableverticesandbruhatcovers}

For a non-empty subword complex $\subwordComplex(\Q,w)$, we connect in this section the cone $\Cplus{w}{\demazure{\Q}}$ to the non-flippable vertices in facets of $\subwordComplex(\Q,w)$.
We then use this description to show in \Cref{prop:nonflipablerootincone} the first containment for \Cref{thm:cone_equality},
\[
  \Cplus{w}{\demazure{\Q}} \ \subseteq \ \bigcap_F \cone\Roots{F},
\]
by showing that all cover labels of atoms in the Bruhat interval $[w, \demazure{\Q}]$ are contained in all cones over root configurations.

\begin{proposition}
  \label{prop:bruhatcoversandrootfunction}
  Let $\subwordComplex(\Q,w)$ be a non-empty subword complex.
  Then
  \[
    \Eplus{w}{\demazure{\Q}} = \bigset{\Root{I}{i}}{I \text{ facet of }\subwordComplex(\Q,w) \text{ and } i \in I \text{ not flippable }}.
  \]
  Moreover, if~$i \in I$ is a flippable index in a facet~$I$ of $\subwordComplex(\Q,w)$ then $\Root{I}{i} \notin \Eplus{w}{\demazure{\Q}}$.
\end{proposition}

First observe that
\begin{align*}
 \Eplus{w}{\demazure{\Q}} \neq \emptyset &\Longleftrightarrow w < \demazure{\Q} \\
                                         &\Longleftrightarrow \subwordComplex(\Q,w) \text{ is a ball } \\
                                         &\Longleftrightarrow \subwordComplex(\Q,w) \text{ has facets containing non-flippable vertices }.
\end{align*}
So the interesting case is if these conditions are satisfied:
For a given $\beta \in \Eplus{w}{\demazure{\Q}}$, \ie, $w \prec s_\beta w \leq \demazure{\Q}$, consider the map
\begin{equation}
\label{eq:iota}
  \iota : \subwordComplex(\Q,s_\beta w) \rightarrow \subwordComplex(\Q, w)
\end{equation}
from facets of $\subwordComplex(\Q,s_\beta w)$ to facets of $\subwordComplex(\Q,w)$ given by $J \mapsto J \cup \{k\}$ where~$k$ is the unique index in the complement of~$J$ such that $\Q_{\{1,\dots,m\} \setminus (J \cup \{k\})}$ is a reduced word for~$w$.
This unique index is well-defined by \Cref{lem:techlem} saying that for any reduced word $\s_1\dots \s_\ell$ for~$s_\beta w$ there is a unique index~$k$ such that $\s_1\dots \widehat \s_k \dots \s_\ell$ is a word of~$w$ and this word is reduced because $w \prec s_\beta w$ implies $\ell(w) = \ell(s_\beta w) - 1$.

This discussion shows that every facet of $\subwordComplex(\Q,s_\beta w)$ is also a face of $\subwordComplex(\Q, w)$, and this face is of codimension~$2$.
It is in particular not surprising that the map~$\iota$ is not injective in general as seen in the following example.

\begin{example}[Type $A_1 \times A_1$]
  We write this example in type $B_3$ to use the above example scheme.
  Let $\Q = 131$ and $w = 13 = 31$.
  The complex $\subwordComplex(\Q,w)$ then contains the two facets
  \[
    \greedy = \{1\} \text{ and } \antigreedy = \{3\}.
  \]
  If we consider the lower cover $s_3 \prec w$ we obtain
  \[
    \iota(\greedy) = \{1,3\} = \iota(\antigreedy).
  \]
\end{example}

\begin{lemma}
\label{lem:nonflippablesarepositive}
  Let $\subwordComplex(\Q,w)$ be a non-empty subword complex, let $I \in \subwordComplex(\Q,w)$ be a facet and let $i \in I$ non-flippable.
  Then $I \setminus \{i\}$ is a facet of $\subwordComplex(\Q,s_\beta w)$ for $\beta = \Root{I}{i}$.
  In particular, the complement of $I \setminus \{i\}$ is a reduced word and $\beta \in \Eplus{w}{\demazure{\Q}}$.
\end{lemma}

\begin{proof}
  For $\Q = \s_1 \dots \s_m$ the complement of the facet $\Q_{\{1,\dots,m\} \setminus I}$ is a reduced word for~$w$.
  We split this word into its prefix $\w_1 = \Q_{\{1,\dots,i\} \setminus I}$ and its suffix $\w_2 = \Q_{\{i+1,\dots,m\} \setminus I}$.
  We then have for the corresponding elements
  \[
    w = w_1 \cdot w_2.
  \]
  Now $i \in I$ is not flippable, hence by \Cref{lem:rootsandflips}\ref{eq:rootsbyflips} we have $\beta = w_1( \alpha_{s_i} ) \in \Phiplus$ and therefore $w_1 \prec w_1\cdot s_i$.
  For $\Q_{rev} = \s_m \dots \s_1$, we have the obvious isomorphism $\subwordComplex(\Q,w) \cong \subwordComplex(\Q_{rev},w^{-1})$ given by $i \mapsto m+1-i$.
  The position $m+1-i$ is thus not flippable in $\subwordComplex(\Q_{rev},w^{-1})$ and we obtain by the same argument that $w_2 \prec s_i\cdot w_2$.
  Applying now \Cref{lem:realtechlem} we obtain 
  \[
    w = w_1\cdot w_2 \prec w_1 \cdot s_i \cdot w_2 = s_\beta w.
  \]
  Hence $\ell(s_\beta w) = \ell(w) + 1$ and we conclude the statement.
\end{proof}

\begin{lemma}
  \label{lem:coversandnonflipables}
  Let $w \prec s_\beta w \leq \demazure{\Q}$.
  Then:
  \begin{enumerate}
    \item\label{it:coversandnonflipables-b} There is a facet $I \in \subwordComplex(\Q,w)$ and an index $i \in I$ with $\Root{I}{i} = \beta$.
    \item\label{it:coversandnonflipables-c} A facet $I \in \subwordComplex(\Q,w)$ is in the image of $\iota$ if and only if $\beta \in \Roots{I}$.
  \end{enumerate}
\end{lemma}

\begin{proof}
  For \ref{it:coversandnonflipables-b}, take any facet $I \in \subwordComplex(\Q,s_\beta w)$ and let~$i_k$ be defined as above with $\iota(I) = I \cup \{i_k\}$.
  We then obtain the subword $\s_{i_1}\dots\s_{i_\ell} = \Q_{\{1,\dots,m\} \setminus I}$ of $\Q$ with
  \[
    s_\beta w = s_{i_1} \dots s_{i_k} \dots s_{i_\ell}, \quad
            w = s_{i_1} \dots \widehat s_{i_k} \dots s_{i_\ell}
  \]
  being reduced expressions for~$s_\beta w$ and for~$w$, respectively.
  This shows that
  \[
    \beta = s_{i_1} \dots s_{i_{k-1}}(\alpha_{s_{i_k}}) = \Root{I \cup \{i_k\}}{i_k}.
  \]
  Clearly~$i_k$ is not flippable in the facet $I \cup \{i_k\}$ of $\subwordComplex(\Q,w)$, as otherwise the given reduced expression of~$s_\beta w$ would contain two different reduced expressions for~$w$, and this is not the case by the uniqueness property in \Cref{lem:techlem}.
  
  For \ref{it:coversandnonflipables-c}, we have already seen in the proof of~\ref{it:coversandnonflipables-b} that $\beta = \Root{I\cup\{i_k\}}{i_k}$ if $I\cup\{i_k\} = \iota(I)$ is in the image of~$\iota$.
  Conversely, let~$I$ be a facet of $\subwordComplex(\Q,w)$ such that $\beta = \Root{I}{i_k}$ for some $i_k \in I$.
  As $w \prec s_\beta w$ especially $\beta \not\in \inv(w)$, hence~$i_k$ is not flippable in~$I$.
  Therefore by \Cref{lem:nonflippablesarepositive} we have $I \setminus \{i_k\}$ is a facet in $\subwordComplex(\Q,s_\beta w)$ and furthermore $I = \iota( I \setminus \{i_k\} )$ is in the image of~$\iota$, which concludes the proof.
\end{proof}

\begin{proof}[Proof of \Cref{prop:bruhatcoversandrootfunction}]
  If $\beta \in \Eplus{w}{\demazure{\Q}}$ by \Cref{lem:coversandnonflipables}\ref{it:coversandnonflipables-b} there is a facet $I \in \subwordComplex(\Q,w)$ and a non-flippable index $i \in I$ with $\Root{I}{i} = \beta$ and thus
  \[
    \Eplus{w}{\demazure{\Q}} \subseteq \bigset{\Root{I}{i}}{I \text{ facet of }\subwordComplex(\Q,w) \text{ and } i \in I \text{ not flippable }}
  \]
  The other inclusion follows immediately with \Cref{lem:nonflippablesarepositive}.
\end{proof}

\begin{example}[Type $B_3$]
  Take $\Q = 123123123$ to be a reduced word for the longest element~$\wo$.
  For the simple generator~$s_1$ we have
  \[
    \Eplus{s_1}{\wo} = \{ \del{010},~ \del{001},~ \del{110} \}
  \]
  corresponding to covers $s_{\del{010}} s_1 = s_2s_1,~ s_{\del{001}} s_1 = s_3s_1$ and $s_{\del{110}} s_1 = s_1s_2$.
  The facets of $\subwordComplex(\Q,s_1)$ and their root configurations are
  \begin{align*}
    \Roots{\{\markedbox{1},2,3,\markedbox{4},5,6,8,9\}} &= \multibracket{ \markedbox{\del{100}},~ \del{010},~ \del{001},~ \markedbox{\del{100}},~ \del{010},~ \del{001},~ \del{110},~ \del{001} } \\
    \Roots{\{\markedbox{1},2,3,5,6,\markedbox{7},8,9\}} &= \multibracket{ \markedbox{\del{100}},~ \del{010},~ \del{001},~ \del{110},~ \del{001},~ \markedbox{\del{\mi{1}00}},~ \del{110},~ \del{001} } \\
    \Roots{\{2,3,\markedbox{4},5,6,\markedbox{7},8,9\}} &= \multibracket{ \del{110},~ \del{001},~ \markedbox{\del{\mi{1}00}},~ \del{110},~ \del{001},~ \markedbox{\del{\mi{1}00}},~ \del{110},~ \del{001} }
  \end{align*}
  with flippable positions in bold red.
\end{example}

For the values of the root function on non-flippable vertices this translates to the following result on root configurations.

\begin{proposition}
  \label{prop:nonflipablerootincone}
  Let $\subwordComplex(\Q,w)$ be a non-empty subword complex.
  Then
  \[
    \Eplus{w}{\demazure{\Q}} \subseteq \cone( \Roots{F} )
  \]
  for every facet $F \in \subwordComplex(\Q,w)$.
\end{proposition}

\begin{proof}
  We prove this statement by induction on the length of $\Q = \s_1\dots\s_m$.
  The statement is trivially true for $m=1$, so we may assume $m \geq 2$.
  Let $\Q' = \s_2 \dots \s_m$ be obtained from $\Q$ by deleting the initial letter $\s_1$.
  Set moreover $\sigma = \demazure{\Q}$ and $\tau = \demazure{\Q'}$.
  One has the decomposition
  \[
    \subwordComplex(\Q,w) = \big(\{1\} \star \subwordComplex(\Q',w)\big) \sqcup \subwordComplex(\Q',s_1 w),
  \]
  where $\subwordComplex(\Q',\s_1 w)$ is only taken into account if $s_1 w \prec w$, compare the proof of~\cite[Theorem~2.5]{KM2004}.
  For a facet $\subwordComplex(\Q,w)$, one either has $1 \in I$ and then $I_1 = I \setminus \{1\} \in \subwordComplex(\Q',w)$ is a facet with $\Roots{I} = \Roots{I_1} \cup \{ \alpha_{s_1} \}$, or $1 \notin I$ and then $I_2 = I \in \subwordComplex(\Q',s_1 w)$ is a facet with $\Roots{I} = s_1(\Roots{I_2})$.
  We aim to show that $\Eplus{w}{\sigma} \subseteq \cone( \Roots{I} )$.

  \textbf{Case 1:} $1 \in I$:
  We may assume by induction that
  \[
    \Eplus{w}{\tau} \subseteq \cone( \Roots{I_1} ).
  \]
  If $\tau = \sigma$, we obtain
  \[
    \Eplus{w}{\sigma} = \Eplus{w}{\tau} \subseteq \cone( \Roots{I_1} ) \subseteq \cone( \Roots{I} ).
  \]
  If $\tau \prec \sigma = s_1 \tau$, we apply \Cref{cor:bruhatcone} to the situation $w \leq \tau \prec \sigma = s_1 \tau$ and obtain
  \[
    \Eplus{w}{\sigma} \subseteq \Cplus{w}{\tau} + \RRpos(\alpha_{s_1}) \subseteq \cone(\Roots{I_1}) + \RRpos(\alpha_{s_1}) = \cone(\Roots{I_1}\cup\{\alpha_{s_1}\}) = \cone(\Roots{I}).
  \]

  \textbf{Case 2:} $1 \notin I$:
  We may assume by induction that
  \[
    \Eplus{s_1 w}{\tau} \subseteq \cone( \Roots{I_2} ), \text{ hence } \Cplus{s_1 w}{\tau} \subseteq \cone( \Roots{I_2} ).
  \]
  Since $1 \notin I$, there is a reduced word for~$w$ that starts with~$\s_1$ and hence $s_1 w \prec w \leq \sigma$.
  By \Cref{lem:extendingcovercones} we obtain
  \[
    s_1\big( \Cplus{w}{\sigma} \big) \subseteq \Cplus{s_1w}{\tau},
  \]
  and thus
  \[
    \Eplus{w}{\sigma} \subseteq \Cplus{w}{\sigma} \subseteq s_1\big( \Cplus{s_1w}{\tau} \big) \subseteq s_1 \cone( \Roots{I_2} ) = \cone( s_1 \Roots{I_2} ) = \cone( \Roots{I} ). \qedhere
  \]
\end{proof}

The following describes a non-trivial example for the situation in \Cref{prop:nonflipablerootincone}.
In particular, it shows that it is not enough to take the root configuration itself in the conclusion of the proposition.

\begin{example}[Type $A_2$]
  Take $\Q = 1212$ with $\demazure{\Q} = 121 = 212$ and let $w = 12$.
  We then have 
  $
    \Eplus{w}{\demazure{\Q}} = \{ \del{01} \},
  $
  the facets of $\subwordComplex(\Q,w)$ and their root configurations are given by
  \[
    \Roots{\{1,2\}} = \multibracket{\del{10},~ \del{01}}, \quad
    \Roots{\{2,3\}} = \multibracket{\del{11},~ \del{\mi{1}0}}, \quad
    \Roots{\{3,4\}} = \multibracket{\del{01},~ \del{\mi{11}} }.
  \]
  While $\del{01} \notin \Roots{\{2,3\}}$ is not in the root configuration, we see $\del{01} = \del{11} + \del{\mi{1}0} \in \cone\Roots{\{2,3\}}$ is in its cone, as proposed.
\end{example}

\subsection{Constructing antigreedy facets inside certain half-spaces}
\label{sec:algorithm}

We call a linear functional $f : V \rightarrow \RR$ \Dfn{non-negative} for the Bruhat interval $[x,y]$ if $f(\beta) \geq 0$ for all $\beta \in \Eplus{x}{y}$.
In this section, we provide an algorithm to construct an antigreedy facet~$I_f$ of $\subwordComplex(\Q,w)$ relative to a given linear functional $f : V \longrightarrow \RR$ that is non-negative for the Bruhat interval $[w,\demazure{\Q}]$.
We show in \Cref{prop:algowelldefined} that the facet~$I_f$ is well-defined and that the linear functional~$f$ is non-negative on its root configuration, $f(\beta) \geq 0$ for all $\beta \in \Roots{I_f}$.
For any vector~$v \notin \Cplus{w}{\demazure{\Q}}$, one may thus choose a linear functional that is non-negative for $[w,\demazure{\Q}]$ while $f(v) < 0$.
We then obtain $v \notin \cone(\Roots{I})$.
This implies the remaining inclusion
\[
  \Cplus{w}{\demazure{\Q}} \ \supseteq \ \bigcap_I \cone\Roots{I},
\]
and thus concludes the proof of \Cref{thm:cone_equality}.

\clearpage
\begin{algo}
\label{algo:uniquefacet}
Computing the $f$-antigreedy facet $I_f$ of the subword complex $\subwordComplex(\Q,w)$

\begin{algorithm}[H]
  \SetKwInOut{Input}{Input}
  \SetKwInOut{Conditions}{Conditions}
  \SetKwInOut{Output}{Output}

  \Input{$\Q = \s_1\dots\s_m$\\ $w \leq \demazure{\Q}$\\ $f: V \rightarrow \RR$ with $f(\beta) \geq 0$ for $\beta \in \Eplus{w}{\demazure{\Q}}$}
  \bigskip
  \Conditions{Conditions (1)--(6) are given by the following decision tree:\\
    
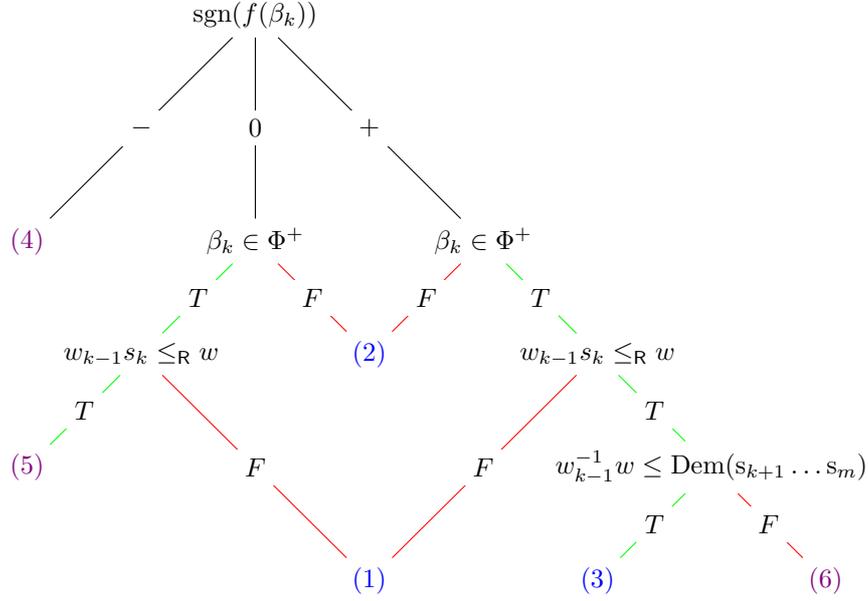
\medskip
    \begin{tikzpicture}[scale=0.5]
      
      \node (f) at (6,15) {$\operatorname{sgn}( f(\beta_k) )$};
      
      \node[violet] (beta1) at (0,9)  {$(4)$};
      \node         (beta2) at (6,9)  {$\beta_k \in \Phiplus$};
      \node         (beta3) at (12,9) {$\beta_k \in \Phiplus$};
      
      \node       (gc1) at (3,6)  {$w_{k-1} s_k \leqweak w$};
      \node[blue] (gc2) at (9,6)  {$(2)$};
      \node       (gc3) at (15,6) {$w_{k-1} s_k \leqweak w$};
      
      \node[violet] (dem1) at (0,3)  {$(5)$};
      \node         (dem2) at (18,3) {$w_{k-1}^{-1} w \leq \demazure{\s_{k+1}\dots\s_m}$};
      
      \node[blue]   (a1) at (9,0)  {$(1)$};
      \node[blue]   (a2) at (15,0) {$(3)$};
      \node[violet] (a3) at (21,0) {$(6)$};
      
      \draw (beta1) -- (f) -- (beta2);
      \draw (beta3) -- (f);
      
      \node[fill=white] (fv1) at (3,12) {$-$};
      \node[fill=white] (fv2) at (6,12) {$0$};
      \node[fill=white] (fv3) at (9,12) {$+$};
      
      \draw[green] (beta2) -- (gc1) -- (dem1);
      \draw[green] (beta3) -- (gc3) -- (dem2) -- (a2);
      
      \draw[red] (beta2) -- (gc2) -- (beta3);
      \draw[red] (gc1) -- (a1) -- (gc3);
      \draw[red] (dem2) -- (a3);
      
      \node[fill=white] (betav1) at (4.5,7.5)  {$T$};
      \node[fill=white] (betav2) at (7.5,7.5)  {$F$};
      \node[fill=white] (betav3) at (10.5,7.5) {$F$};
      \node[fill=white] (betav4) at (13.5,7.5) {$T$};
      
      \node[fill=white] (lwv1) at (1.5,4.5)  {$T$};
      \node[fill=white] (lwv2) at (6,3)      {$F$};
      \node[fill=white] (lwv3) at (12,3)     {$F$};
      \node[fill=white] (lwv4) at (16.5,4.5) {$T$};
      
      \node[fill=white] (bwv1) at (16.5,1.5) {$T$};
      \node[fill=white] (bwv2) at (19.5,1.5) {$F$};
      
    \end{tikzpicture} \\
    The sign can be positive ($+$), negative ($-$), or zero ($0$). \\
    The statements can be true ($T$) or false ($F$).}
  \bigskip
  \Output{$I_f \subset \{1,\dots,m\}$}
  \bigskip

  $w_0              \leftarrow \one \in W$\\
  $I_0 \hspace*{3pt}\leftarrow \{\}$\\
  \bigskip

  \For{$k = 1,\dots, m$}{
    $\beta_k        \leftarrow w_{k-1}(\alpha_{s_k})$

    \If{Condition (1) or (2) or (3)}{
      $I_k \hspace*{3pt}\leftarrow I_{k-1} \cup \{k\}$ \\
      $w_k              \leftarrow w_{k-1}$
    }
    \ElseIf{Condition (4) or (5) or (6)}{
      $I_k \hspace*{3pt}\leftarrow I_{k-1}$ \\
      $w_k              \leftarrow w_{k-1} \cdot s_k$
    }
  }
  $I_f \leftarrow I_m$ \\
  \Return $I_f$
\end{algorithm}
\end{algo}

\bigskip

We aim to prove the following properties of the output set $I_f \subseteq \{1,\dots,m\}$ of \Cref{algo:uniquefacet}.

\begin{theorem}
\label{prop:algowelldefined}
  Let $\subwordComplex(\Q,w)$ be a non-empty subword complex and let $f: V \rightarrow \RR$ with $f(\beta) \geq 0$ for $\beta \in \Eplus{w}{\demazure{\Q}}$.
  The output set~$I_f \subseteq \{1,\dots,m\}$ of \Cref{algo:uniquefacet} has the following properties:
  \begin{enumerate}
    \item\label{it:algowelldefined-a} $I_f$ is a facet of the subword complex $\subwordComplex(\Q,w)$, \ie, the word $\Q_{\{1,\dots,m\} \setminus I_f}$ is a reduced word for~$w$.
    \item\label{it:algowelldefined-b} For~$i \in I_f$, we have $f(\Root{I_f}{i}) \geq 0$.
    \item\label{it:algowelldefined-c} For~$i \in I_f$ with $f(\Root{I_f}{i}) = 0$ and $\Root{I_f}{i} \not\in \Eplus{w}{\demazure{\Q}}$, we have $\Root{I_f}{i} \in \Phiminus$.
  \end{enumerate}
\end{theorem}

\begin{remark}
  Applying this algorithm for $w \leq \demazure{\Q}$ and a linear functional $f : V \rightarrow \RR$ which is positive on
  \begin{itemize}
    \item the basis $\Delta$ of~$V$, \ie, $f(\alpha_s) > 0$ for all $s \in \sref$, yields the greedy facet~$\greedy$,
    \item the basis $w(\Delta)$ of~$V$, \ie, $f(w(\alpha_s)) > 0$ for all $s \in \sref$, yields the antigreedy facet~$\antigreedy$,
  \end{itemize}
  of $\subwordComplex(\Q,w)$ as seen in \Cref{lem:greedyfacet}.
\end{remark}

In the $k$-th step of \Cref{algo:uniquefacet} we have to choose whether to extend the output set~$I_f$ or to apply the simple reflection~$s_k$ to the element~$w_{k-1}$.
In the former, we aim to ensure that a reduced word for~$w$ still can be obtained after this step.
The following lemma formalizes this condition using Bruhat order.

\begin{lemma}
  \label{lem:extensionofreducedwords}
  Let $\Q = \s_1 \dots \s_m$ and $w \leq \demazure{\Q}$ be an element in~$W$.
  Furthermore let 
  \[
    v = \wordprod{\Q}{\{1,\dots,i\} \setminus X} \leq w
  \]
  for some subset  $X \subseteq \{1,\dots,i\}$ such that $\Q_{\{1,\dots,i\} \setminus X}$ is a reduced word for~$v$.
  Then~$\Q_{\{1,\dots,i\} \setminus X}$ can be extended to a reduced word for~$w$ by a subword of~$\Q_{\{i+1,\dots,m\}}$ if and only if
  \[
    v^{-1} w \leq \demazure{\Q_{\{i+1,\dots,m\}}}.
  \]
\end{lemma}

\begin{proof}
  This is a direct consequence of the definition of the Demazure product in terms of Bruhat order as given in~\eqref{eq:demazurestrongdef}.
\end{proof}

The crucial step in the proof of \Cref{prop:algowelldefined} is the following property of the algorithm.

\begin{proposition}
\label{prop:algoprops}
  At the end of the $k$-th iteration of the \emph{for loop}, the word $\Q_{\{1,\dots,k\}\setminus I_k}$ is a reduced word for~$w_k$ and can be extended to a reduced word for~$w$ by a subword of $\Q_{\{k+1\,\dots,m\}}$.
\end{proposition}

\begin{proof}
  Let $\Q = \s_1 \dots \s_m$.
  \Cref{lem:extensionofreducedwords} implies that if $\Q_{ \{1,\dots,k\} \setminus I_k}$ is a reduced word for~$w_k$, then it can be extended to a reduced word for~$w$ by a subword of~$\Q_{\{k+1,\dots,m\}}$ if and only if
  \begin{equation}\label{eq:algoprops}
    w_k^{-1} w \leq \demazure{\Q_{\{k+1,\dots,m\}}}.\tag{$\ast$}
  \end{equation}
  We prove in parallel by induction on the loop index~$k$ that the word $\Q_{\{1,\dots,k\}\setminus I_k}$ is reduced and that property~\eqref{eq:algoprops} holds.
  This is, we show that if these two properties are satisfied for all indices $t < k$, than these are also satisfied for the index~$k$.
  Since the two properties are satisfied before the first loop starts (this is, for $k=0$), the statement then follows.

  \medskip

  For step~$k\geq 1$ we assume by induction that $\Q_{\{1,\dots,k-1\} \setminus I_{k-1}}$ is a reduced word for~$w_{k-1}$ and that $(w_{k-1})^{-1} w \leq \demazure{\Q_{\{k,\dots,m\}}}$.

  \medskip

  We first show that the word $\Q_{\{1,\dots,k-1\} \setminus I_{k-1}}$ is reduced.
  In the case of conditions $(1),(2)$ and $(3)$, we have $\Q_{\{1,\dots,k\} \setminus I_{k}} = \Q_{\{1,\dots,k-1\} \setminus I_{k-1}}$ is reduced.
  In the case of conditions $(5)$ and $(6)$, we have $w_{k-1}( \alpha_{s_k} ) = \beta_k \in \Phiplus$ hence $w_k = w_{k-1}\cdot s_k$ is reduced.
  It remains to show the case of condition~$(4)$, \ie, $f(\beta_k) < 0$, where we need to show $\beta_k \in \Phiplus$.

  By contraposition, we may assume $\beta_k \in \Phiminus$ and have to show that $f(\beta_k) \geq 0$.
  As~$w_{k-1}$ is reduced and $\beta_k = w_{k-1}(\alpha_{s_k})$ we have $-\beta_k \in \inv(w_{k-1})$.
  Therefore there is an index $t<k$ such that
  \[
    w_{t-1}(\alpha_{s_t}) = \beta_t = -\beta_k
  \]
  and at step~$t$ we were in either of the cases of conditions $(4),(5)$ or $(6)$.
  If we were in the cases of conditions $(4)$ or $(5)$, then $f(\beta_t) \leq 0$ and hence $f(\beta_k) \geq 0$ as desired.

  We conclude that proof of reducedness by showing that it was impossible to be in the case of condition~$(6)$ at step~$t$.
  By induction, the reduced word $\Q_{\{1,\dots,k-1\} \setminus I_{k-1}}$ for~$w_{k-1}$ can be extended in $\Q_{\{k,\dots,m\}}$ to a reduced word for~$w$.
  Since $\beta_k \in \Phiminus$, this extended reduced word cannot use the letter~$\s_k$.
  Thus there exists a facet $I \in \subwordComplex(\Q,w)$ such that
  \[
    \{1,\dots,k\} \cap I = I_{k-1} \cup \{k\}.
  \]
  But as $-\beta_k \in \inv(w_{k-1}) \subseteq \inv(w)$, \Cref{lem:rootsandflips} implies that the index~$k \in I$ is flippable to~$t$ in this facet.
  In other words, the flip of~$k$ in~$I$ gives the facet $(I \setminus \{k\}) \cup \{t\}$.
  Now this in particular implies that $\Q_{\{t+1,\dots,m\}}$ contains a reduced word for $w_{t-1}^{-1} w$, hence
  \[
    w_{t-1}^{-1} w \leq \demazure{\Q_{\{t+1,\dots,m\}}}
  \]
  and we were not in the case of condition~$(6)$.
  
  \medskip
  
  We conclude with showing that also \eqref{eq:algoprops} holds for the index~$k$.
  If we are in one of the cases of conditions $(4),(5)$ or $(6)$, then \eqref{eq:algoprops} holds by~\eqref{eq:demazurestrongdef}.
  If we are in the case of condition~$(3)$, we have $w_k = w_{k-1}$ and \eqref{eq:algoprops} is part of the condition.
  
  If we are in the case of condition~$(2)$, we have $\beta_k = w_{k-1}(\alpha_{s_k}) \in \Phiminus$.
  But as we have the decomposition
  \begin{equation}
    \label{eq:algoprops2}
    \inv(w) = \inv(w_{k-1}) \sqcup w_{k-1} \cdot \inv(w_{k-1}^{-1} w), \tag{$\ast\ast$}
  \end{equation}
  we have $\alpha_{s_k} \not\in \inv(w_{k-1}^{-1} w)$, hence no reduced word for $w_{k-1}^{-1} w$ starts with~$\s_k$.
  Therefore $w_k^{-1} w = w_{k-1}^{-1} w \leq \demazure{\Q_{\{k+1,\dots,m\}}}$.

  If we are in the case of condition~$(1)$, we have $w_{k-1} s_k \not\leqweak w$, hence $\beta_k \not\in \inv(w)$.
  Thus again by \eqref{eq:algoprops2} we get $\alpha_{s_k} \not\in \inv(w_{k-1}^{-1} w)$, hence no reduced word for $w_{k-1}^{-1} w$ starts with~$\s_k$ and the same argument applies.
\end{proof}

\begin{proof}[Proof of \Cref{prop:algowelldefined}]
  Applying \Cref{prop:algoprops} for $k=m$, this is at the end of the last iteration of the \emph{for loop}, the word $\Q \setminus I_m$ is a reduced word for~$w = w_m$.
  This implies \ref{it:algowelldefined-a}.

  For \ref{it:algowelldefined-b} it is immediate from the definition that the root configuration of the facet~$I_f$ is
  \[
    \Roots{I_f} = \multiset{\beta_k}{k \in I_f}.
  \]
  The algorithm only adds an index~$k$ to the facet if and only if one of the conditions $(1),(2)$ or $(3)$ is fulfilled.
  Since these conditions all contain the condition $f(\beta_k) \geq 0$, we conclude \ref{it:algowelldefined-b}.

  For \ref{it:algowelldefined-c} we show that if $f(\beta_k) = 0$ and $\beta_k \in \Phiplus$ for some $k \in I_f$, then $\beta_k \in \Eplus{w}{\demazure{\Q}}$.
  This means that adding~$k$ to~$I_f$ was done in the case of condition~$(1)$.
  We thus have $w_{k-1} s_k \not\leqweak w$.
  As seen before this implies $\beta_k \not\in \inv(w)$, which means that the index~$k$ is not flippable in~$I_f$.
  We conclude the proof by invoking \Cref{prop:bruhatcoversandrootfunction}.
\end{proof}

\begin{example}[Type $B_2$]
  Let $\Q = 21122112$ with $\demazure{\Q} = 1212 = 2121 = \wo$ and let $w = 12$.
  Let furthermore $f : V \rightarrow \RR$ be the linear functional sending $\del{10}$ to~$-2$ and $\del{01}$ to~$1$.
  We check that this linear functional is non-negative on
  \[
  \Eplus{w}{\demazure{\Q}} = \{ \del{01}, \del{12} \}.
  \]
  The steps in \Cref{algo:uniquefacet} are then shown in the following table:
  \begin{center}
    \resizebox{\textwidth}{!}{
      \begin{tabular}{lll}
        $m=1: \ \beta_1 = \del{01}, \ f(\beta_1) > 0, \ \beta_1 \in \Phiplus, \ s_2 \not\leqweak w$ & $\leadsto (1): \ I_1 = \{1\},$ & $w_1 = e,$ \\

        $m=2: \ \beta_2 = \del{10}, \ f(\beta_2) < 0$ &  $\leadsto (4): \ I_2 = \{1\},$ & $w_2 = s_1,$ \\

        $m=3: \ \beta_3 = \del{\mi{1}0}, \ f(\beta_3) > 0, \ \beta_3 \in \Phiminus$ &  $\leadsto (2): \ I_3 = \{1,3\},$ & $w_3 = s_1,$ \\

        $m=4: \ \beta_4 = \del{11}, \ f(\beta_4) < 0$ &  $\leadsto (4): \ I_4 = \{1,3\},$ & $w_4 = s_1s_2,$ \\

        $m=5: \ \beta_5 = \del{\mi{11}}, \ f(\beta_5) > 0, \ \beta_5 \in \Phiminus$ &  $\leadsto (2): \ I_5 = \{1,3,5\},$ & $w_5 = s_1s_2,$ \\

        $m=6: \ \beta_6 = \del{12}, \ f(\beta_6) = 0, \ \beta_6 \in \Phiplus, \ s_1s_2s_1 \not\leqweak w$ &  $\leadsto (1): \ I_6 = \{1,3,5,6\},$ & $w_6 = s_1s_2,$ \\

        $m=7: \ \beta_7 = \del{\mi{11}}, \ f(\beta_7) > 0, \ \beta_7 \in \Phiminus$ &  $\leadsto (2): \ I_7 = \{1,3,5,6,7\},$ & $w_7 = s_1s_2,$ \\

        $m=8: \ \beta_8 = \del{12}, \ f(\beta_8) = 0, \ \beta_8 \in \Phiplus, \ s_1s_2s_1 \not\leqweak w$ &  $\leadsto (1): \ I_8 = \{1,3,5,6,7,8\},$ & $w_8 = s_1s_2$. \\ \\
      \end{tabular}
    }
  \end{center}
  The algorithm finally returns $I_f = I_8 = \{1,3,5,6,7,8\}$ with root configuration
  \[
  \Roots{I_f} = \multibracket{\del{01},~ \del{\mi{1}0},~ \del{\mi{11}},~ \del{12},~ \del{\mi{11}},~ \del{12} },
  \]
  and \Cref{prop:algowelldefined}\ref{it:algowelldefined-b} and \ref{it:algowelldefined-c} are both satisfied, as expected.
\end{example}

We prove in the following section, see \Cref{thm:uniquefantigreedyfacet}, that the facet~$I_f$ is uniquely determined by the two conditions given in \Cref{prop:algowelldefined}\ref{it:algowelldefined-b} and \ref{it:algowelldefined-c}.

\subsection{Uniqueness of $f$-antigreedy facets}
\label{sec:uniqueantigreedyfacets}

We show in this section that the $f$-antigreedy facet~$I_f$ in \Cref{algo:uniquefacet} is unique in the following sense.

\begin{theorem}
\label{thm:uniquefantigreedyfacet}
  Let~$I$ be a facet of the non-empty subword complex $\subwordComplex(\Q,w)$ and let $f : V \rightarrow \RR$ be a linear functional that is non-negative for the Bruhat interval $[w,\demazure{\Q}]$.
  If the facet~$I$ satisfies the conclusions in \Cref{prop:algowelldefined}\ref{it:algowelldefined-b} and \ref{it:algowelldefined-c}, then $I = I_f$ is the facet produced by \Cref{algo:uniquefacet}.
\end{theorem}

Before proving this theorem, we also provide one important corollary and recast the two properties \Cref{prop:algowelldefined}\ref{it:algowelldefined-b} and \ref{it:algowelldefined-c} into subword complex terms as follows.
For a facet $I \in \subwordComplex(\Q,w)$ and $i \in I$ flippable we call the flip of~$i \in I$ \Dfn{$f$-preserving} if $f(\Root{I}{i}) = 0$, and denote by
\[
  \subwordComplex_f(\Q,w) = \bigset{I \text{ facet of } \subwordComplex(\Q,w)}{\forall i \in I : f( \Root{I}{i} ) \geq 0}
\]
the set of facets whose root configuration is contained in the closed positive halfspace defined by~$f$.
For such a facet~$I$ of $\subwordComplex_f(\Q,w)$ denote moreover by
\[
  \RootsPos{I}{f} = \set{i \in I}{f( \Root{I}{i} ) > 0}
\]
those positions in~$I$ for which the root function is $f$-positive.
A facet $I \in \subwordComplex(\Q,w)$ then satisfies \Cref{prop:algowelldefined}\ref{it:algowelldefined-b} if and only if $I \in \subwordComplex_f(\Q,w)$ and it satisfies \Cref{prop:algowelldefined}\ref{it:algowelldefined-c} if and only if
\[
  \set{\Root{I}{i}}{i \in I \setminus \RootsPos{I}{f}} \subseteq \Phiminus \cup \Eplus{w}{\demazure{\Q}}.
\]

\begin{proposition}
  Let $\subwordComplex(\Q,w)$ be a non-empty subword complex and let $f : V \rightarrow \RR$ be a linear functional.
  Then
  \[
    \subwordComplex_f(\Q,w) \text{ is non-empty } \Leftrightarrow f \text{ is non-negative on } [w,\demazure{\Q}].
  \]
\end{proposition}

\begin{proof}
  If~$f$ is non-negative on $[w,\demazure{\Q}]$ then the facet $I_f$ generated by \Cref{algo:uniquefacet} is in $\subwordComplex_f(\Q,w)$ by \Cref{prop:algowelldefined}.
  Otherwise if~$f$ is not non-negative, \Cref{prop:nonflipablerootincone} ensures that every facet $I \in \subwordComplex(\Q,w)$ contains at least one position $i \in I$ for that $f( \Root{I}{i} ) < 0$ and thus $\subwordComplex_f(\Q,w)$ is empty.
\end{proof}

The following crucial corollary of \Cref{thm:uniquefantigreedyfacet} generalizes \cite[Conjecture~7.1]{Pilaud-Stump-2015}.

\begin{corollary}
\label{prop:connectedcomp}
  Let $\subwordComplex(\Q,w)$ be a non-empty subword complex and let $f : V \rightarrow \RR$ be a linear functional which is non-negative on $[w,\demazure{\Q}]$.
  Then $\subwordComplex_f(\Q,w)$ forms a connected component of the graph of $f$-preserving flips in $\subwordComplex(\Q,w)$ and moreover,
  \[
    \subwordComplex_f(\Q,w) \cong \subwordComplex(\Q_{\{1,\dots,m\} \setminus \RootsPos{I}{f}}, w)
  \]
  for any facet $I \in \subwordComplex_f(\Q,w)$.
\end{corollary}

In the remainder of this section, we prove \Cref{thm:uniquefantigreedyfacet} and \Cref{prop:connectedcomp}, and start with collecting several technical lemmas.

\begin{lemma}
  \label{lem:q1notflipable}
  Let $\subwordComplex(\Q,w)$ be a non-empty subword complex with $\Q = \s\ \s_2\dots\s_m$.
  The following are equivalent:
  \begin{itemize}
    \item There exist facets $I,J \in \subwordComplex(\Q,w)$ with $1 \in I$ and $1 \not\in J$.
    \item There exists a facet $K \in \subwordComplex(\Q_{\{2,\dots,m\}}, s w)$ with $\alpha_{s} \in \Roots{K}$, and $s w \prec w$.
  \end{itemize}
\end{lemma}

\begin{proof}
  Let $\Q' = \s_2\dots\s_m$, and observe that there exists a facet~$I$ with $1 \in I$ if and only if $\subwordComplex(\Q',w)$ is not empty, and there exists a facet~$J$ with $1 \notin J$ if and only if $sw \prec w$.

  Assuming first that these two properties hold, the existence of a facet~$K$ of $\subwordComplex(\Q',sw)$ with $\alpha_{s} \in \Roots{K}$ is ensured by \Cref{prop:bruhatcoversandrootfunction} as $\alpha_s \in \inv(w)$.

  Let now be $K \in \subwordComplex(\Q', s w)$ with $\alpha_{s} \in \Roots{K}$, and let $sw \prec w$.
  Again by \Cref{prop:bruhatcoversandrootfunction}, we obtain that the word $\Q'$ contains a reduced word for $w = ssw$.
  Since this is equivalent to $\subwordComplex(\Q',w)$ being not empty, it concludes the proof.
\end{proof}

\begin{lemma}
  \label{lem:coneandfunctional}
  Let $\Q = \s\ \s_2 \dots \s_m$ and $w \leq \demazure{\Q}$.
  The following are equivalent:
  \begin{itemize}
    \item For every linear functional~$f$ with $f( \alpha_{s} ) > 0$, the complex $\subwordComplex_{f \circ s}(\Q_{\{2,\dots,m\}}, s w)$ is empty.
    \item For every facet $K \in \subwordComplex(\Q_{\{2,\dots,m\}}, s w)$ we have $\alpha_{s} \in \cone\Roots{K}$.
  \end{itemize}
\end{lemma}

\begin{proof}
  The first item says that for any functional $f : V \rightarrow \RR$ with $f( \alpha_{s} ) > 0$ and every facet $K \in \subwordComplex(\Q_{\{2,\dots,m\}}, sw)$, there is a $\beta \in \Roots{K}$ such that $(f \circ s) ( \beta ) < 0$.
  By replacing~$f$ with $(f \circ s)$ this says for any functional $f: V \rightarrow \RR$ with $f( \alpha_{s} ) < 0$ and every facet $K \in \subwordComplex(\Q{\{2,\dots,m\}}, s w)$, there is a $\beta \in \Roots{K}$ such that $f ( \beta ) < 0$.
  Now by duality of cones and functionals, this says that $\alpha_{s} \in \cone( \Roots{ K } )$ for every facet $K \in \subwordComplex(\Q_{\{2,\dots,m\}}, s w)$.
\end{proof}

\begin{corollary}
  \label{cor:noposfacet}
  Let $\Q = \s\ \s_2 \dots \s_m$, $w \leq \demazure{\Q}$ and $f : V \rightarrow \RR$ be a linear functional with $f( \alpha_s ) > 0$.
  Let furthermore $I,J \in \subwordComplex(\Q,w)$ be two facets with $1 \in I$ and   $1 \not\in J$.
  Then
  \begin{itemize}
    \item For every facet $K \in \subwordComplex(\Q_{\{2,\dots,m\}}, sw)$ we have $\alpha_s \in \cone( \Roots{K} )$, and
    \item the complex $\subwordComplex_{f \circ s}(\Q_{\{2,\dots,m\}}, sw)$ is empty.
  \end{itemize}
\end{corollary}

\begin{proof}
  First by \Cref{lem:q1notflipable} we have $s w \prec w$.
  Applying \Cref{prop:nonflipablerootincone} we get $\alpha_{s} \in \cone( \Roots{K} )$ for every facet $K \in \subwordComplex(\Q_{\{2,\dots,m\}}, sw)$ and by \Cref{lem:coneandfunctional} we then get $\subwordComplex_{f \circ s}(\Q_{\{2,\dots,m\}}, s w)$ is empty.
\end{proof}

We remark that in the situations of \Cref{lem:coneandfunctional} and \Cref{cor:noposfacet} we have $sw \prec w$. Thus the functional $f \circ s$ is not non-negative for the Bruhat interval $[sw, \demazure{\Q}]$.
Thus, the conclusion
\[
  \subwordComplex_{f \circ s}(\Q_{\{2,\dots,m\}}, sw) = \emptyset
\]
is in agreement with \Cref{prop:connectedcomp}.

\begin{proof}[Proof of \Cref{thm:uniquefantigreedyfacet}]
  Let $I_f$ be the facet of $\subwordComplex(\Q,w)$ produced by \Cref{algo:uniquefacet} and let $I \neq I_f$ be another facet.
  Let~$k$ be the smallest index for which~$I$ differs from $I_f$, \ie, the smallest index such that either $k \in I \setminus I_f$ or $k \in I_f \setminus I$.

  If the algorithm was violated at step~$k$ in condition~$(4)$, the facet~$I$ violates \Cref{prop:algowelldefined}\ref{it:algowelldefined-b}, and if the algorithm was violated in condition~$(5)$, the facet~$I$ violates \Cref{prop:algowelldefined}\ref{it:algowelldefined-c}.

  If the algorithm was violated in step~$k$ in either of the conditions~$(1), (2)$ and~$(6)$, the partially constructed set could not be extended to a facet of $\subwordComplex(\Q,w)$ which cannot happen since~$I$ is chosen to be such a facet. 

  For the remaining case, \ie, that the algorithm was violated at step~$k$ in condition~$(3)$, we assume without loss of generality that $k=1$.
  Hence, we have $f(\beta_1) = f(\alpha_{s}) > 0$, $s w \prec w$ and $w \leq \demazure{\Q_{\{2,\dots,m\}}}$.
  By the latter there is a facet $J \in \subwordComplex(\Q,w)$ with $1 \in J$, but as $sw \prec w$ we have $\alpha_s \in \inv(w)$, hence $1 \in J$ is flippable.
  Let~$I'$ be the adjacent facet, such that $I' \setminus \{i\} = J \setminus \{1\}$.
  Then by \Cref{cor:noposfacet} we have $\subwordComplex_{f \circ s}(\Q_{\{2,\dots,m\}}, sw) = \emptyset$.
  Now facets $ K' \in \subwordComplex(\Q_{\{2,\dots,m\}}, sw)$ are in one-to-one correspondence with facets $K' = K \in \subwordComplex(\Q,w)$ with $1 \not\in K$.
  Furthermore we have $\Roots{K} = s \Roots{K'}$ for every such facet, hence if we violated \Cref{algo:uniquefacet} at step~$k$ in condition~$(3)$ by not adding~$1$ to $I_{k+1}$, we obtain $I \not\in \subwordComplex_f(\Q,w)$ and conclude the statement.
\end{proof}

\begin{proof}[Proof of \Cref{prop:connectedcomp}]
  \label{pf:pf_connectedcomp}
  Let $I,J \in \subwordComplex(\Q,w)$ be two adjacent facets with $I \setminus \{i\} = J \setminus \{j\}$ and $f( \Root{I}{i} ) = 0 = f( \Root{J}{j} )$, \ie, the flip from~$I$ to $J$ is $f$-preserving.
  By \Cref{lem:rootsandflips} we have for all $k \in [m]$ that $\Root{J}{k} - \Root{I}{k} \in \RR \cdot \Root{I}{i}$ and thus $f( \Root{I}{k} ) = f( \Root{J}{k} )$.
  Therefore $\subwordComplex_f(\Q,w)$ is closed under $f$-preserving flips, and also the set of $f$-positive indices is preserved, $\RootsPos{I}{f} = \RootsPos{J}{f}$.
  By performing $f$-preserving flips we can reach a facet $K \in \subwordComplex_f(\Q,w)$ with $\RootsPos{K}{f} = \RootsPos{I}{f}$ and $(K \setminus \RootsPos{K}{f}) \setminus \Eplus{w}{\demazure{\Q}} \subseteq \Phiminus$.
  By \Cref{thm:uniquefantigreedyfacet} this facet is unique, hence $\subwordComplex_f(\Q,w)$ is connected via $f$-preserving flips.

  For the second part, recall that we proved in the first part that $\RootsPos{I}{f}$ is independent of the facet $I \in \subwordComplex_{f}(\Q,w)$.
  By inserting the positions $\RootsPos{I}{f}$ into $\Q_{\{1,\dots,m\} \setminus \RootsPos{I}{f}}$, we obtain a natural isomorphism
  \[
    \subwordComplex(\Q_{\{1,\dots,m\} \setminus \RootsPos{I}{f}}, w) \cong \bigset{I \in \subwordComplex(\Q,w)}{\RootsPos{I}{f} \subseteq I}.
  \]
  In particular, $\subwordComplex_{f}(\Q,w)$ is a subset of the right hand side.
  Now let $J \in \subwordComplex(\Q,w)$ be any facet with $\RootsPos{I}{f} \subseteq J$.
  By the above isomorphism, the facets~$J$ and~$I$ are connected by a sequence of flips not containing positions in $\RootsPos{I}{f}$.
  Thus, the sequence of flips from~$I$ to~$J$ not containing positions in $\RootsPos{I}{f}$ is $f$-preserving and we obtain $J \in \subwordComplex_{f}(\Q,w)$.
\end{proof}

\section{Brick polyhedra for subword complexes}
\label{sec:brickpolytopes}

Based on \Cref{prop:connectedcomp}, we extend in this section brick polytopes that were developed in~\cite{Pilaud-Stump-2015} for root-independent spherical subword complexes towards brick polyhedra for general subword complexes.

\medskip

We first recall some elementary notions for polyhedra in the vector space~$V$.
A subset~$P \subseteq V$ is called \Dfn{polyhedron} if there are finitely many linear functionals $f_1,\dots,f_k : V \rightarrow \RR$ and scalars $b_1,\dots,b_k \in \RR$ such that
\[
  P = \bigset{v \in V}{f_i(v) + b_i \geq 0 \text{ for all } 1 \leq i \leq k}.
\]
A bounded polyhedron is furthermore called \Dfn{polytopal}.
A \Dfn{face} of a polyhedron~$P$ is a subset $F \subseteq P$ such that there exists a linear functional $f : V \rightarrow \RR$ and a scalar $b \in \RR$ with $f + b$ is non-negative on~$P$ and such that $F = \set{v \in P}{f(v) + b = 0}$.
The linear functional~$f$ is called \Dfn{defining functional} for the face~$F$ and the pair $(f,b)$ is called \Dfn{defining hyperplane}.
The \Dfn{local cone} of a polyhedron~$P$ at a point~$q \in P$ is the cone over~$P$ seen from the point~$q$,
\[
  \pcone{P}{q} = \cone\set{p-q}{p \in P}.
\]
Moreover, the \Dfn{(inner) normal cone}~$\normalcone{F}$ of a face $F \subseteq P$ is the cone of defining functionals of~$F$ and the \Dfn{(inner) normal fan}~$\normalfan{P}$ is the collection of normal cones of faces of~$P$,
\[
  \normalfan{P} = \bigset{\normalcone{F}}{F \text{ a face of } P}.
\]
Finally, the \Dfn{Minskowski sum} $P+Q$ of two polyhedra~$P$ and~$Q$ is given by pointwise vector addition.
Observe that every $d$-dimensional face of $P+Q$ is given by the Minskowski sum of an $i$-dimensional face of~$P$ and a $j$-dimensional face of~$Q$ with $i+j = d$.

\medskip

We now define the brick polyhedron of a non-empty subword complex $\subwordComplex(\Q,w)$ with $\Q =\s_1 \dots \s_m$.
Following \cite[Definition 4.1]{Pilaud-Stump-2015}, the \Dfn{weight function} $\Weight{I}{\cdot} : [m] \rightarrow W(\nabla) \subseteq V$ of a facet~$I$ of $\subwordComplex(\Q,w)$ is defined by
$
  \Weight{I}{k} = \wordprod{\Q}{\{1,\dots,k-1\} \setminus I}(\omega_{s_k}).
$ 
and the \Dfn{brick vector}\footnote{In comparision with the original definition of brick polytopes in~\cite{Pilaud-Stump-2015}, we introduce a minus sign here. This is done in order to simplify notations in later statements.} of~$I$ is then
\[
  \brickVector{I} = - \sum\limits_{k=1}^{m} \Weight{I}{k}.
\]

The following definition generalizes the definition of brick polytopes for spherical subword complexes.

\begin{definition}
\label{def:brickpolytope}
  The \Dfn{brick polyhedron} of a non-empty subword complex $\subwordComplex(\Q,w)$ is the Minskowski sum of the convex hull of all brick vectors and the Bruhat cone $\Cplus{w}{\demazure{\Q}}$,
  \[
    \brickPolytope(\Q,w) = \conv\bigset{\brickVector{I}}{I \text{ facet of } \subwordComplex(\Q,w)} + \Cplus{w}{\demazure{\Q}}. 
  \]
\end{definition}

This definition immediately implies that every brick vector is contained in the brick polyhedron and that every vertex of the brick polyhedron is a brick vector.
We moreover record that brick polyhedra indeed generalize the previously known notion of brick polytopes up to a switch in the sign of brick vectors.

\begin{proposition}
  The brick polyhedron $\brickPolytope(\Q,w)$ of a non-empty subword complex $\subwordComplex(\Q,w)$ is polytopal if and only if $\subwordComplex(\Q,w)$ is spherical.
  In this case, the brick polyhedron is the convex hull of all brick vectors.
\end{proposition}

\begin{proof}
  This follows from the observation made in \Cref{sec:subwordcomplexes} that the subword complex $\subwordComplex(\Q,w)$ is spherical if and only if $\demazure{\Q} = w$ if and only if $\Cplus{w}{\demazure{\Q}} = \{0\}$.
\end{proof}

\begin{example}[Type $A_2$]
  \label{ex:brickpoly}
  Let $\Q = 11212$ and $w = 12$.
  We then have the facets of $\subwordComplex(\Q,w)$ and the brick polyhedron $\brickPolytope(\Q,w)$ given by
  
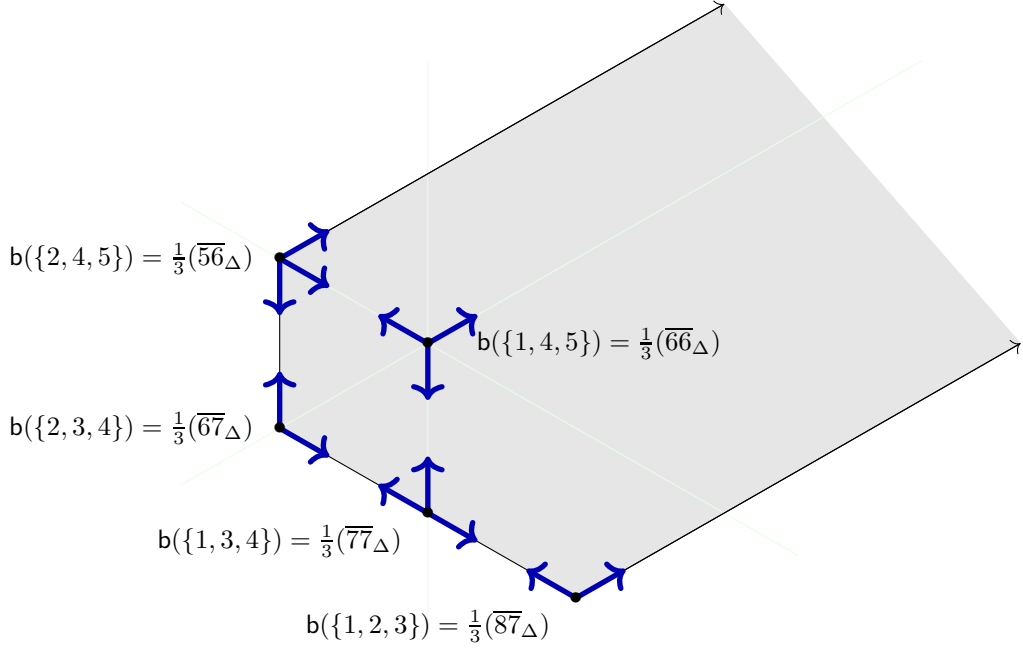
\[
  \begin{tikzpicture}[scale=1.5]
  \draw[fill=grey!20, grey!20] (2.598,3) -- (-1.299,0.75) -- (-1.299,-0.75) -- (1.299,-2.25) -- (5.196,0) -- (2.598,3);
  
  \draw[green!10] (-2.165,-1.25) -- (4.33,2.5);
  \draw[green!10] (-2.165,1.25) -- (3.2475,-1.875);
  \draw[green!10] (0,-2.5) -- (0,2.5);
  
  \node (a1) at (0,-2.5)     {$\brickVector{\{1,2,3\}} = \tfrac{1}{3}(\del{\mi{87}})$};
  \node (a2) at (-1.3,-1.75) {$\brickVector{\{1,3,4\}} = \tfrac{1}{3}(\del{\mi{77}})$};
  \node (a3) at (1.5,0)      {$\brickVector{\{1,4,5\}} = \tfrac{1}{3}(\del{\mi{66}})$};
  \node (a4) at (-2.6,-0.75) {$\brickVector{\{2,3,4\}} = \tfrac{1}{3}(\del{\mi{67}})$};
  \node (a5) at (-2.6,0.75)  {$\brickVector{\{2,4,5\}} = \tfrac{1}{3}(\del{\mi{56}})$};
  
  \draw (2.598,3) -- (-1.299,0.75) -- (-1.299,-0.75) -- (1.299,-2.25) -- (5.196,0);
  \draw[-{>[scale=1.2]}] (-1.299,0.75) -- (2.598,3);
  \draw[-{>[scale=1.2]}] (1.299,-2.25) -- (5.196,0);
  
  \draw[darkblue, line width = 2pt, ->] (1.299,-2.25) -- (1.732,-2);
  \draw[darkblue, line width = 2pt, ->] (1.299,-2.25) -- (0.866,-2);
  
  \draw[darkblue, line width = 2pt, ->] (0,-1.5) -- (-0.433,-1.25);
  \draw[darkblue, line width = 2pt, ->] (0,-1.5) -- (0.433,-1.75);
  \draw[darkblue, line width = 2pt, ->] (0,-1.5) -- (0,-1);
  
  \draw[darkblue, line width = 2pt, ->] (0,0) -- (0,-0.5);
  \draw[darkblue, line width = 2pt, ->] (0,0) -- (0.433,0.25);
  \draw[darkblue, line width = 2pt, ->] (0,0) -- (-0.433,0.25);
  
  \draw[darkblue, line width = 2pt, ->] (-1.299,-0.75) -- (-1.299,-0.25);
  \draw[darkblue, line width = 2pt, ->] (-1.299,-0.75) -- (-0.866,-1);
  
  \draw[darkblue, line width = 2pt, ->] (-1.299,0.75) -- (-1.299,0.25);
  \draw[darkblue, line width = 2pt, ->] (-1.299,0.75) -- (-0.866,1);
  \draw[darkblue, line width = 2pt, ->] (-1.299,0.75) -- (-0.866,0.5);
  
  \node (b1) at (1.299,-2.25) {$\bullet$};
  \node (b2) at (0,-1.5) {$\bullet$};
  \node (b3) at (0,0) {$\bullet$};
  \node (b4) at (-1.299,-0.75) {$\bullet$};
  \node (b5) at (-1.299,0.75) {$\bullet$};
  
  \end{tikzpicture}
  \]
  with arrows pointing towards the respective root configurations
  \begin{gather*}
  \Roots{\{1,2,3\}} = \{\del{10},~ \del{10},~ \del{01}\}, \quad 
  \Roots{\{1,3,4\}} = \{\del{10},~ \del{11},~ \del{\mi{10}}\}, \\
  \Roots{\{1,4,5\}} = \{\del{10},~ \del{01},~ \del{\mi{11}}\}, \quad 
  \Roots{\{2,3,4\}} = \{\del{\mi{10}},~ \del{11},~ \del{\mi{10}}\}, \\
  \Roots{\{2,4,5\}} = \{\del{\mi{10}},~ \del{01},~ \del{\mi{11}}\}.
  \end{gather*}
\end{example}

\subsection{Local cones of brick polyhedra at brick vectors}

The definition of brick polyhedra is justified by the following generalization of \cite[Proposition~4.7]{Pilaud-Stump-2015}.

\begin{theorem}
  \label{cor:vertexpointedcone}
  The local cone of the brick polyhedron $\brickPolytope(\Q,w)$ at the brick vector $\brickVector{I}$ coincides with the cone generated by the root configuration of the facet~$I$ of $\subwordComplex(\Q,w)$.
  In symbols,
  \[
  \pcone{\brickPolytope(\Q,w)}{\brickVector{I}} = \cone \Roots{I}.
  \]
  In particular, the brick vector $\brickVector{I}$ is a vertex of $\brickPolytope(\Q,w)$ if and only if $\Roots{I}$ is pointed.
\end{theorem}

Based on \Cref{thm:cone_equality}, we obtain the following equivalent description of brick polyhedra.

\begin{corollary}
  \label{cor:brickpolyhdescription}
  We have
  \[
  \brickPolytope(\Q,w) =  \bigcap_{I \text{ facet of } \subwordComplex(\Q,w)} \big( \brickVector{I} +  \cone \Roots{I} \big).
  \]
\end{corollary}

A linear functional $f : V \rightarrow \RR$ is a defining functional for the non-empty brick polyhedron $\brickPolytope(\Q,w)$ if and only if it is non-negative on $[w,\demazure{\Q}]$.
For such a defining functional~$f$ with corresponding defining hyperplane $(f,b)$, denote by $B_f = \set{ v \in P}{f(v) + b = 0}$ the corresponding face of $\brickPolytope(\Q,w)$.
Recall also the facet~$I_f$ produced by \Cref{algo:uniquefacet} in \Cref{sec:algorithm} and from \Cref{sec:uniqueantigreedyfacets} that $\subwordComplex_f(\Q,w)$ is the set of facets~$I$ of $\subwordComplex(\Q,w)$ with $f(\Root{I}{i}) \geq 0$ for all $i \in I$.

\medskip

The following statement then generalizes \cite[Lemma~4.6]{Pilaud-Stump-2015}.

\begin{proposition}
\label{prop:facesofbrickpolyarebrickpoly}
  Let $f : V \rightarrow \RR$ be a linear functional which is non-negative on $[w,\demazure{\Q}]$.
  For a facet $I \in \subwordComplex(\Q,w)$, we have
  \[
    \brickVector{I} \in B_f \Longleftrightarrow I \in \subwordComplex_{f}(\Q,w).
  \]
\end{proposition}

For the proof we first recall the following two lemmas from \cite[Section~4]{Pilaud-Stump-2015} that were stated for root-independent spherical subword complexes.
These remain valid in the present context with the same proofs as given.

\begin{lemma}{\cite[Lemma~4.4]{Pilaud-Stump-2015}}
  \label{lem:weightssandflips}
  Let $\Q = \s_1 \dots \s_m$ and let~$I$ be a facet of the non-empty subword complex $\subwordComplex(\Q,w)$.
  \begin{enumerate}
    \item If $\s_k = \s_{k+1}$ we have  $\Weight{I}{k+1} = 
    \begin{cases}
      \Weight{I}{k}               & \text{ if } k \in I, \\
      \Weight{I}{k} - \Root{I}{k} & \text{ if } k \notin I.
    \end{cases}$
    \item\label{eq:weightsbyflips2} If~$J$ is an adjacent facet with $I \setminus \{i\} = J \setminus \{j\}$, then $\Weight{J}{\cdot}$ is obtained from $\Weight{I}{\cdot}$ by:
    \[
      \Weight{J}{k} = \begin{cases}
                        s_{\Root{I}{i}}( \Weight{I}{k} ) & \text{ if } \min(i,j) < k \leq \max(i,j), \\
                        \Weight{I}{k} & \text{ otherwise}.
                      \end{cases}
    \]
    \item\label{eq:weightsbyflips3} For $j \notin I$, we have $\langle \Root{I}{j}, \Weight{I}{k} \rangle$ is non-negative if $j \geq k$, and non-positive if $j < k$.  
  \end{enumerate}
\end{lemma}

The following lemma is a consequence of \Cref{lem:weightssandflips}.
Due to our sign-switch in the definition of the brick vector and in comparison to the original statement we also have a switch of the direction here.

\begin{lemma}{\cite[Lemma~4.5]{Pilaud-Stump-2015}}
  \label{lem:differenceofbrickvectors}
  If~$I$ and~$J$ are two facets in $\subwordComplex(\Q,w)$ with $I \setminus \{i\} = J \setminus \{j\}$, then the difference of the brick vectors $\brickVector{J} - \brickVector{I}$ is a positive multiple of $\Root{I}{i}$.
\end{lemma}

\begin{proof}[Proof of \Cref{prop:facesofbrickpolyarebrickpoly}]
  Assume first that $\brickVector{I} \in B_f$.
  For every index $i \in I$ we then either have $\Root{I}{i} \in \Eplus{w}{\demazure{\Q}}$ or the index~$i$ is flippable.
  In the former case we have $f(\Root{I}{i}) \geq 0$ as~$f$ is non-negative on $[w,\demazure{\Q}]$.
  In the latter case, let~$J$ be the adjacent facet with $I \setminus \{i\} = J \setminus \{j\}$.
  Then $\brickVector{J} - \brickVector{I}$ is a positive multiple of $\Root{I}{i}$ by \Cref{lem:differenceofbrickvectors} and thus $f(\Root{I}{i}) \geq 0$.
  We thus obtain $I \in \subwordComplex_{f}(\Q,w)$.

  Assume now that $I \in \subwordComplex_{f}(\Q,w)$.
  Let~$K$ be any facet such that $\brickVector{K} \in B_f$.
  It follows from the first part of the proof  that $K \in \subwordComplex_{f}(\Q,w)$, and then from \Cref{prop:connectedcomp} that~$K$ and~$I$ are connected via $f$-preserving flips.
  Again by~\Cref{lem:differenceofbrickvectors}, we obtain $f(\brickVector{I}) = f(\brickVector{K})$ and thus $\brickVector{I} \in B_f$.
\end{proof}

Before proving \Cref{cor:vertexpointedcone}, we collect the following consequence of \Cref{prop:facesofbrickpolyarebrickpoly}.

\begin{corollary}
  \label{cor:edgesareflips}
  Any two facets~$I$ and~$J$ of $\subwordComplex(\Q,w)$ whose brick vectors $\brickVector{I}, \brickVector{J}$ are contained in an edge $E \subseteq \brickPolytope(\Q,w)$ are connected by a flip.
\end{corollary}

\begin{proof}
  Let~$f$ be a defining functional for the edge~$E = B_f$.
  \Cref{prop:facesofbrickpolyarebrickpoly} then implies that the facets $I,J$ are contained in $\subwordComplex_{f}(\Q,w)$.
  By the isomorphism in \Cref{prop:connectedcomp}, this complex corresponds to a Coxeter system $W'$ of rank~$1$ where any two facets are connected by a flip.
\end{proof}

The following example shows that there are also flips between facets with pointed root configurations that are not edges of the brick polyhedron.

\begin{example}[Type $A_2$]
  \label{ex:innerpointedflip}
  The spherical subword complex $\subwordComplex(211221,121)$ has the eight facets
  \[
  \begin{array}{cccc}
  I_1 = \{1,2,4\},& I_2 = \{1,2,5\},& I_3 = \{1,3,4\},& I_4 = \{1,3,5\}, \\
  I_5 = \{2,4,6\},& I_6 = \{2,5,6\},& I_7 = \{3,4,6\},& I_8 = \{3,5,6\}.
  \end{array}
  \]
  After a shift by $3(\omega_1 + \omega_2)$, we have the polytopal brick polyhedron  $\brickPolytope(\Q,w)$ with flip graph
  
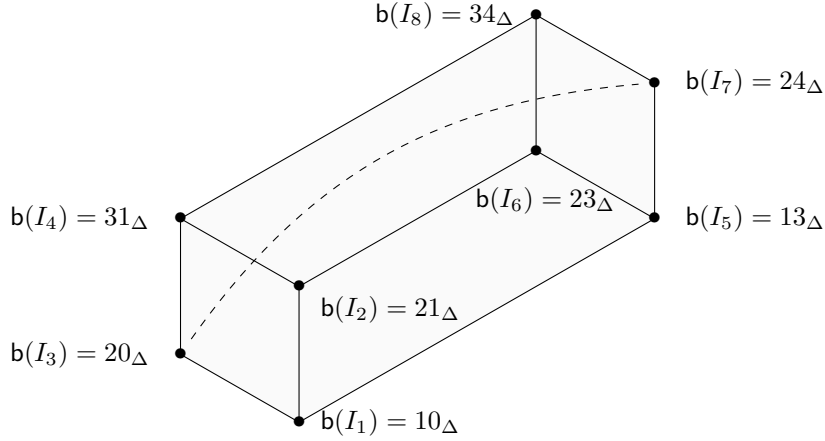
\[
  \begin{tikzpicture}[scale=1.2]
  \draw[fill=black!2] (-1.299,0.75) -- (2.598,3) -- (2.598,4.5) -- (1.299,5.25) -- (-2.598,3) -- (-2.598,1.5) -- (-1.299,0.75);
  
  \draw (-1.299,0.75) -- (-1.299,2.25) -- (-2.598,3);
  \draw (-1.299,2.25) -- (1.299,3.75) -- (1.299,5.25);
  \draw (1.299,3.75) -- (2.598,3);
  
  \node (a1) at (-1.299,0.75) {$\bullet$};
  \node (a2) at (-1.299,2.25) {$\bullet$};
  \node (a3) at (-2.598,1.5) {$\bullet$};
  \node (a4) at (-2.598,3) {$\bullet$};
  \node (a5) at (2.598,3) {$\bullet$};
  \node (a6) at (1.299,3.75) {$\bullet$};
  \node (a7) at (2.598,4.5) {$\bullet$};
  \node (a8) at (1.299,5.25) {$\bullet$};
  
  \draw[dashed] (a3) to[bend left=25] (a7);
  
  \node (b1) at (-0.3,0.75) {$\brickVector{I_1} = \del{10}$};
  \node (b2) at (-0.3,2) {$\brickVector{I_2} = \del{21}$};
  \node (b3) at (-3.7,1.5)  {$\brickVector{I_3} = \del{20}$};
  \node (b4) at (-3.7,3)    {$\brickVector{I_4} = \del{31}$};
  \node (b5) at (3.7,3)     {$\brickVector{I_5} = \del{13}$};
  \node (b6) at (1.4,3.2)  {$\brickVector{I_6} = \del{23}$};
  \node (b7) at (3.7,4.5)   {$\brickVector{I_7} = \del{24}$};
  \node (b8) at (0.3,5.25)  {$\brickVector{I_8} = \del{34}$};
  \end{tikzpicture}
  \]
  The facets~$I_3$ and~$I_7$ have pointed root configurations and are connected by a flip which does not correspond to an edge of the brick polyhedron.
\end{example}

\begin{proof}[Proof of \Cref{cor:vertexpointedcone}]
  Fix a facet~$I$ of $\subwordComplex(\Q,w)$.
  Denote by $C_B(I)$ the local cone of the brick polyhedron at the brick vector $\brickVector{I}$.
  We first show $\Roots{I} \subseteq C_B(I)$.
  Let $i \in I$ and $\beta = \Root{I}{i}$.
  If $\beta \in \Eplus{w}{\demazure{\Q}}$ we have
  \[
    \beta \in \Cplus{w}{\demazure{\Q}} \subseteq C_B(I).
  \]
  Otherwise~$i$ is flippable to some facet~$J$, \ie, there is a facet~$J$ and an index~$j \in J$ such that $I \setminus \{i\} = J \setminus \{j\}$.
  We then have $\brickVector{J} - \brickVector{I}$ is a positive multiple of $\beta$ by \Cref{lem:differenceofbrickvectors}.
  As $\brickVector{J} \in \brickPolytope(\Q,w)$ we conclude $\beta \subseteq C_B(I)$.

  The other inclusion $C_B(I) \subseteq \cone\Roots{I}$ is obvious if $\cone\Roots{I} = V$.
  So assume otherwise, let~$v \in V \setminus \cone\Roots{I}$ and let~$f$ be a linear functional which is negative on~$v$ and non-negative on $\Roots{I}$.
  It follows from \Cref{thm:cone_equality} that~$f$ is then also non-negative on $[w,\demazure{\Q}]$.
  By definition, $I \in \subwordComplex_f(\Q,w)$ and \Cref{prop:facesofbrickpolyarebrickpoly} ensures that $\brickVector{I} \in B_f$.
  This yields that $C_B(I)$ is non-negative on~$f$ and we conclude that $v \notin C_B(I)$.
\end{proof}

\begin{remark}
  We have seen in \Cref{prop:facesofbrickpolyarebrickpoly} that brick vectors contained in a face~$B_f$ of the brick polyhedron $\brickPolytope(\Q,w)$ are in one-to-one correspondence with facets in $\subwordComplex_f(\Q,w)$.
  We furthermore have by \Cref{prop:connectedcomp} the identification
  \[
    \subwordComplex_f(\Q,w) \cong \subwordComplex(\Q_{\{1,\dots,m\} \setminus \RootsPos{I_f}{f}}, w),
  \]
  and \Cref{cor:vertexpointedcone} ensures that~$B_f$ and $\brickPolytope(\Q_{\{1,\dots,m\} \setminus \RootsPos{I_f}{f}}, w)$ have the same \emph{local structure}:
  \begin{itemize}
    \item The direction of flips between brick vectors is preserved,
    \item Local cones in $\brickPolytope(\Q_{\{1,\dots,m\} \setminus \RootsPos{I_f}{f}}, w)$ agree with those inside the face~$B_f$ of $\brickPolytope(\Q, w)$,
    \item The normal fans of~$B_f$ and of $\brickPolytope(\Q_{\{1,\dots,m\} \setminus \RootsPos{I_f}{f}}, w)$ coincide, and
    \item $B_f$ is polytopal if and only if $\subwordComplex(\Q_{\{1,\dots,m\} \setminus \RootsPos{I_f}{f}}, w)$ is spherical.
  \end{itemize}
  Nevertheless, one may check that~$B_f$ and $\brickPolytope(\Q_{\{1,\dots,m\} \setminus \RootsPos{I_f}{f}}, w)$ do not necessarily coincide.
  Take the edge connecting~$\brickVector{I_1}$ and~$\brickVector{I_5}$ in \Cref{ex:innerpointedflip}.
  The defining functional~$f$ is zero on $\del{01}$ and positive on $\del{10}$ and $\del{11}$ and we obtain
  \[
    \subwordComplex_f(\Q,w) = \{ I_1 = \{1,2,4\},~ I_5 = \{2,4,6\} \} \ \text{ and } \ \RootsPos{I_1}{f} = \RootsPos{I_5}{f} = \{ 2, 4\}.
  \]
  The reduced subword complex then is $\subwordComplex_f(\Q,w) \cong \subwordComplex(2121,121)$ with facets
  \[
  J_1 = \{ 1 \} \ \text{ and } \ J_2 = \{ 4 \},
  \]
  and brick vectors $\brickVector{J_1} = \del{\mi{12}}$ and $\brickVector{J_2} = \del{\mi{10}}$.
  We thus see
  \[
  \brickVector{I_5} - \brickVector{I_I} = \del{03} \ \text{ and } \ \brickVector{J_2} - \brickVector{J_1} = \del{02},
  \]
  hence this edge has different lengths in the two brick polytopes. 
\end{remark}

\subsection{Normal fans of brick polyhedra from Coxeter fans}
\label{sec:furtherproperties}

\Cref{cor:vertexpointedcone} makes it possible to generalize further properties developed for brick polytopes in \cite{Pilaud-Stump-2015} to brick polyhedra for general subword complexes.
Several proofs in this section are similar to those given in \cite{Pilaud-Stump-2015}.
Define the \Dfn{Coxeter fan} of~$W$ as
\[
  \coxeterfan{W} = \bigset{w(\cone\nabla')}{w \in W,\ \nabla' \subseteq \nabla}
\]
with \Dfn{fundamental chamber} $\mathcal{C} = \cone(\nabla)$ being the cone generated by the fundamental weights.
The aim of this section is to describe how to glue together and delete chambers in the Coxeter fan to obtain the normal fan of the brick polyhedron .
To this end, associate to a Bruhat interval $[x,y]$ a (lower) order ideal in the weak order by
\begin{align*}
  \idweak{x,y} & = \bigset{w \in W}{\Eplus{x}{y} \subseteq w(\Phiplus)} \\
               & = \bigset{w \in W}{\inv(w) \cap \Eplus{x}{y} = \emptyset}.
\end{align*}
This is indeed a lower order ideal as for $w \in \idweak{x,y}$ and $z \leqweak w$ we have $\inv(z) \subseteq \inv(w)$ and thus $z \in \idweak{x,y}$.

\begin{proposition}
  Let $\subwordComplex(\Q,w)$ be a non-empty subword complex and let $z \in W$.
  Then there exists a facet~$I$ such that $\Roots{I} \subseteq z(\Phiplus)$ if and only if $z \in \idweak{w,\demazure{\Q}}$.
  In this case, the facet~$I$ is uniquely given by the facet $I_f$ produced by \Cref{algo:uniquefacet} for the linear functional~$f$ which is positive on $z(\Phiplus)$ and negative on $z(\Phiminus)$.
\end{proposition}

\begin{proof}
  It follows from \Cref{thm:cone_equality} that for $z \notin \idweak{w,\demazure{\Q}}$ there does not exist a facet~$I$ such that $\Roots{I} \subseteq z(\Phiplus)$.
  Now let $z \in \idweak{w,\demazure{\Q}}$ and let~$f$ be a linear functional as in the statement.
  Then $\Eplus{w}{\demazure{\Q}} \subseteq z(\Phiplus)$ ensures that~$f$ is positive for the Bruhat interval $[w,\demazure{\Q}]$.
  \Cref{thm:uniquefantigreedyfacet} then gives that the facet~$I_f$ is the unique facet with $f$-positive root configuration.
  In other words, $I_f$ is the unique facet for which $\Roots{I} \subseteq z(\Phiplus)$.
\end{proof}

For a non-empty subword complex $\subwordComplex(\Q,w)$, this proposition allows to define a map
\[
  \kappa : \idweak{w,\demazure{\Q}} \rightarrow \subwordComplex(\Q,w)
\]
by sending $z \in \idweak{w,\demazure{\Q}}$ to the unique facet~$I_f$ with $\Roots{I} \subseteq z(\Phiplus)$ where~$f$ and~$I_f$ are given as in the proposition.

\begin{lemma}
  Let $\subwordComplex(\Q,w)$ be non-empty and $x,y \in \idweak{w,\demazure{\Q}}$.
  If $\kappa(x) = \kappa(y)$ then
  \[
    \kappa(x) = \kappa(w) = \kappa(y) \text{ for all } w \in \weakint{x}{y}.
  \]
\end{lemma}

\begin{proof}
  Let $\kappa(x) = I = \kappa(y)$ and consider $w \in \weakint{x}{y}$.
  Recall that for any $z \in W$, we have
  \[
    z(\Phiplus) \cap \Phiplus = \Phiplus \setminus \inv(z), \quad
    z(\Phiplus) \cap \Phiminus = -\inv(z)\ .
  \]
  With $\kappa(x) = I = \kappa(y)$ and $\inv(x) \subseteq \inv(w) \subseteq \inv(y)$, we obtain
  \[
    \Roots{I} \cap \Phiplus \subseteq \Phiplus \setminus \inv(y) \subseteq \Phiplus \setminus \inv(w), \quad
    \Roots{I} \cap \Phiminus \subseteq -\inv(x) \subseteq -\inv(w).
  \]
  The statement follows.
\end{proof}

\begin{example}[Type $A_2$]
  We continue \Cref{ex:brickpoly} with $\Q = 11212$ and $w = 12$.
  We then have $\Eplus{12}{121} = \{\del{01}\}$ and $\idweak{12,121} = \{ e, 1, 12 \}$.
  The brick polyhedron with normal fan is
  \[
    \scalebox{0.9}
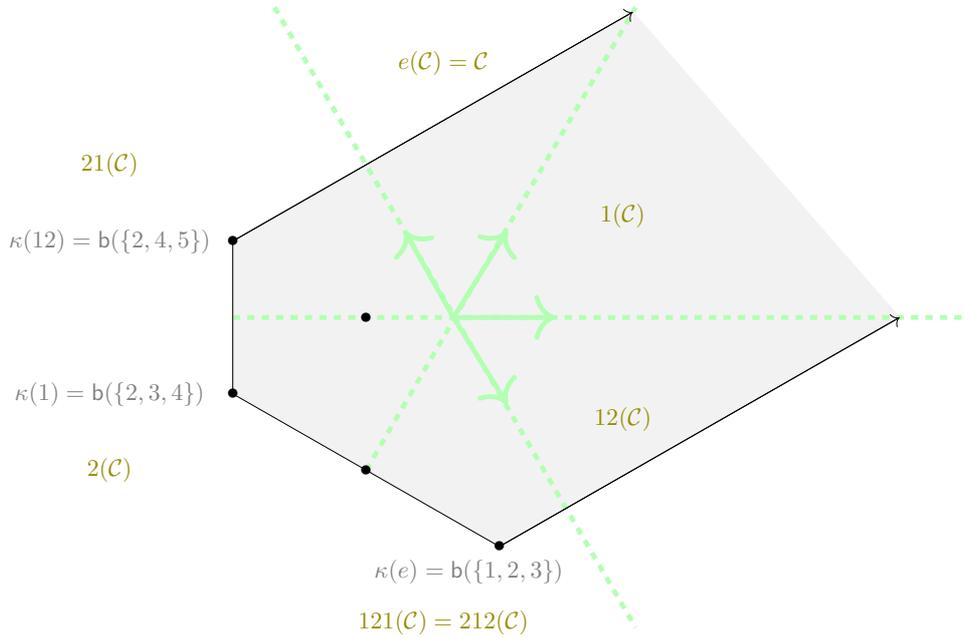
{
    \begin{tikzpicture}[scale=1.5]
      \draw[fill=grey!10, grey!10] (2.598,3) -- (-1.299,0.75) -- (-1.299,-0.75) -- (1.299,-2.25) -- (5.196,0) -- (2.598,3);
      
      \draw[green!30, line width=2pt, dashed] (0,-1.5) -- (2.625,3.0465);
      \draw[green!30, line width=2pt, dashed] (-0.8927,3.0465) -- (2.625,-3.0465);
      \draw[green!30, line width=2pt, dashed] (-1.299,0) -- (6,0);
      
      \draw[green!30, line width=2pt, -{>[scale=1.5]}] (0.8505,0) -- (1.8505,0);
      \draw[green!30, line width=2pt, -{>[scale=1.5]}] (0.8505,0) -- (0.375,0.8505);
      \draw[green!30, line width=2pt, -{>[scale=1.5]}] (0.8505,0) -- (1.375,0.8505);
      \draw[green!30, line width=2pt, -{>[scale=1.5]}] (0.8505,0) -- (1.375,-0.8505);

      \node[grey] (a1) at (1,-2.5) {$\kappa(e) = \brickVector{\{1,2,3\}}$};
      \node[grey] (a4) at (-2.5,-0.75) {$\kappa(1) = \brickVector{\{2,3,4\}}$};
      \node[grey] (a5) at (-2.5,0.75)  {$\kappa(12) = \brickVector{\{2,4,5\}}$};
      
      \draw (2.598,3) -- (-1.299,0.75) -- (-1.299,-0.75) -- (1.299,-2.25) -- (5.196,0);
      \draw[-{>[scale=1.2]}] (-1.299,0.75) -- (2.598,3);
      \draw[-{>[scale=1.2]}] (1.299,-2.25) -- (5.196,0);
      
      \node (b1) at (1.299,-2.25) {$\bullet$};
      \node (b2) at (0,-1.5) {$\bullet$};
      \node (b3) at (0,0) {$\bullet$};
      \node (b4) at (-1.299,-0.75) {$\bullet$};
      \node (b5) at (-1.299,0.75) {$\bullet$};
      
      \node[olive] (c1) at (2.5,1) {$1(\mathcal{C})$};
      \node[olive] (c2) at (2.5,-1) {$12(\mathcal{C})$};
      \node[olive] (c3) at (0.75,2.5) {$e(\mathcal{C}) = \mathcal{C}$};
      \node[olive] (c4) at (0.75,-3) {$121(\mathcal{C}) = 212(\mathcal{C})$};
      \node[olive] (c5) at (-2.5,1.5) {$21(\mathcal{C})$};
      \node[olive] (c6) at (-2.5,-1.5) {$2(\mathcal{C})$};
    \end{tikzpicture}
    }
  \]
  In particular we see that
  \begin{gather*}
    e(\mathcal{C}) \text{ is the inner normal cone of } \kappa(e) = \{1,2,3\}, \\
    1(\mathcal{C}) \text{ is the inner normal cone of } \kappa(1) = \{2,3,4\}, \\
    12(\mathcal{C}) \text{ is the inner normal cone of } \kappa(12) = \{2,4,5\},
  \end{gather*}
  and the union $2(\mathcal{C}) \cup 21(\mathcal{C}) \cup 212(\mathcal{C})$ is the set of linear functionals not non-negative on $\Cplus{12}{121}$.
\end{example}

\begin{example}[Type $B_2$]
  Let $Q = 2221$ and $w = 2$.
  Then $\demazure{\Q} = 12$ and $\Eplus{2}{21} = \{ \del{12} \}$.
  Furthermore we have that $\idweak{2,21} = \{ e, 1, 12, 2 \}$ is an order ideal that is not an interval.
  In $\subwordComplex(\Q,w)$ we have the three facets with brick vectors
  \[
    \brickVector{\{1,2,4\}} = \tfrac{1}{2}(\del{\mi{58}}),\quad
    \brickVector{\{1,3,4\}} = \tfrac{1}{2}(\del{\mi{57}}),\quad
    \brickVector{\{2,3,4\}} = \tfrac{1}{2}(\del{\mi{56}}),
  \]
  that form an edge.
  The brick polyhedron is thus
  \[
      \brickPolytope(\Q,w) = \conv\big\{ \tfrac{1}{2}(\del{\mi{58}}),\ \tfrac{1}{2}(\del{\mi{56}}) \big\} + \cone\big\{ \del{12} \big\}.
  \]
  We furthermore have
  \[
    \kappa(e) = \kappa(1) = \kappa(12) = \{1,2,4\} \ \text{ and } \ \kappa(2) = \{2,3,4\}
  \]
  as shown in the weak order of type~$B_2$:
  
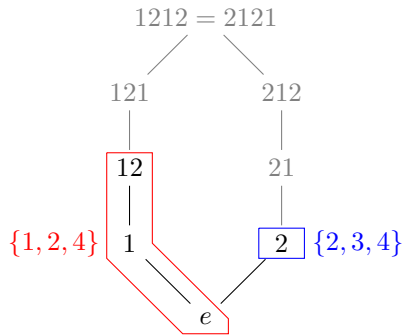
\[
    \begin{tikzpicture}[scale=1.0]
      \node (e) at (0,0) {$e$};
      \node (1) at (-1,1) {$1$};
      \node (2) at (1,1) {$2$};
      \node (12) at (-1,2) {$12$};
      \node[grey] (21) at (1,2) {$21$};
      \node[grey] (121) at (-1,3) {$121$};
      \node[grey] (212) at (1,3) {$212$};
      \node[grey] (1212) at (0,4) {$1212=2121$};
      
      \draw (12) -- (1) -- (e) -- (2);
      \draw[grey] (12) -- (121) -- (1212) -- (212) -- (21) -- (2);
      
      \draw[red] (-1.3,2.2) -- (-0.7,2.2) -- (-0.7,1) -- (0.3,0) -- (0.3,-0.2) -- (-0.3,-0.2) -- (-1.3,0.8) -- (-1.3,2.2);
      
      \draw[blue] (0.7,1.2) -- (1.3,1.2) -- (1.3,0.8) -- (0.7,0.8) -- (0.7,1.2);
      
      \node[red] (i1) at (-2,1) {$\{1,2,4\}$};
      \node[blue] (i2) at (2,1) {$\{2,3,4\}$};
    \end{tikzpicture}
  \]
\end{example}

\begin{proposition}
  \label{lem:kappasurjective}
  The map~$\kappa$ maps surjectively onto the facets of $\subwordComplex(\Q,w)$ with pointed root configurations.
\end{proposition}

\begin{proof}
  Since $\Roots{\kappa(z)} \subseteq z(\Phiplus)$, the facet $\kappa(z)$ has a pointed root configuration by construction.
  For a facet~$I$ with pointed root configuration, let~$f$ be any linear functional that is positive on~$\Roots{I}$ and non-zero on all roots.
  Then there is a unique element $z \in W$ with $z(\Phiplus) = \set{\beta \in \Phi}{f(\beta) > 0}$ and conclude that $\kappa(z) = I$.
\end{proof}

The main theorem of this section is the following generalization of \cite[Proposition 5.4]{Pilaud-Stump-2015}.

\begin{theorem}
  \label{thm:brickpolynormalcone}
  Let $\brickVector{I}$ be a vertex of $\brickPolytope(\Q,w)$, \ie, $I \in \subwordComplex(\Q,w)$ is a facet with pointed root configuration.
  The (closure of the) normal cone $\normalcone{\brickVector{I}}$ is the union of the chambers $z( \mathcal{C} )$ of $\coxeterfan{W}$ given by the elements $z \in W$ with $\kappa(z) = I$.
\end{theorem}

We also get the following generalization of \cite[Corollary ~5.5]{Pilaud-Stump-2015}.

\begin{corollary}
\label{thm:coarseningcoxeter}
  The normal fan $\normalfan{\brickPolytope(\Q,w)}$ is obtained from the Coxeter fan by glueing together the chambers corresponding to fibers of the map~$\kappa$, and deleting the chambers corresponding to elements in~$W$ not in $\idweak{w,\demazure{\Q}}$.
\end{corollary}

The crucial parts of the proof of \Cref{thm:brickpolynormalcone} is extracted into the following two lemmas.
The first generalizes \cite[Lemma~5.3]{Pilaud-Stump-2015}.

\begin{lemma}
\label{lem:weakcoversunderkappa}
  Let $z \in \idweak{w,\demazure{\Q}}$ and $s \in \sref$ such that $z s \in \idweak{w,\demazure{\Q}}$.
  Then $\kappa(zs)$ is obtained from $\kappa(z)$ as follows:
  \begin{itemize}
    \item If $z(\alpha_s) \in \Roots{\kappa(z)}$, then $\kappa(zs)$ is obtained from $\kappa(z)$ by flipping the unique index $i \in I$ such that $\Root{I}{i} = z(\alpha_s)$ and the obtained facet has again a pointed root configuration.
    \item If $z(\alpha_s) \notin \Roots{\kappa(z)}$, then $\kappa(zs) = \kappa(z)$.
  \end{itemize}
\end{lemma}

\begin{proof}
  Set $\alpha = \alpha_s$ and $I = \kappa(z)$, and observe that the root $z(\alpha)$ generates a ray of $z(\cone(\Phiplus))$.
  Either $z(\alpha) \in \Phiminus$ or $z(\alpha) \in \inv(zs)$.
  Since $zs \in \idweak{w,\demazure{\Q}}$ we obtain $z(\alpha) \notin \Eplus{w}{\demazure{\Q}}$.

  \medskip

  Let $z(\alpha) \notin \Roots{I}$.
  Then $\Roots{I} \subseteq z(\Phiplus) \cap zs(\Phiplus)$ and we obtain $\kappa(sz) = I$.

  Let $z(\alpha) \in \Roots{I}$.
  Since $\Roots{I} \subseteq z(\Phiplus)$, $z(\alpha)$ generates a ray of the (pointed) cone over $\Roots{I}$.
  Therefore, there exists a linear functional~$f$ with $f(z(\alpha)) = 0$ and~$f$ positive on $z(\Phiplus) \setminus \{ z(\alpha) \}$ (and in particular on $\Roots{I}$).
  Since $z(\alpha) \notin \Eplus{w}{\demazure{\Q}}$, the face $B_f \subseteq \brickPolytope(\Q,w)$ is a polytopal edge.
  \Cref{cor:vertexpointedcone} and \Cref{prop:facesofbrickpolyarebrickpoly} ensure that $\brickVector{I}$ is a vertex of $B_f$.
  Let~$J$ be the facet for which $\brickVector{J}$ is the other vertex of the edge~$B_f$.
  \Cref{cor:edgesareflips} then shows that~$I$ and~$J$ are connected by a flip.
  If $z(\alpha) \in \Phiplus$, then~$J$ is obtained from~$I$ by flipping the smallest index $i \in I$ for which $\Root{I}{i} = z(\alpha)$, and if $z(\alpha) \in \Phiminus$, then~$J$ is obtained from~$I$ by flipping the largest index $i \in I$ for which $\Root{I}{i} = z(\alpha)$.
\end{proof}

\begin{lemma}
  \label{lem:rootconfunderkappa}
  Let $\brickVector{I}$ be a vertex of $\brickPolytope(\Q,w)$.
  Then 
  \[
    \cone(\Roots{I}) = \bigcap\limits_{z \in \kappa^{-1}(I)} z(\cone(\Phiplus) ).
  \]
\end{lemma}

\begin{proof}
  The definition of the map~$\kappa$ immediately implies
  \[
    \cone(\Roots{I}) \subseteq \bigcap\limits_{z \in \kappa^{-1}(I)} z( \cone(\Phiplus) ).
  \]
  For the other inclusion let $\beta \in \Phi$ be a root that generates a ray of $\cap_{z \in \kappa^{-1}(I)} z( \cone(\Phiplus) )$.
  Then~$\beta$ is also ray of $u(\cone(\Phiplus))$ for some particular $u \in \kappa^{-1}(I)$.
  Thus $\beta = u(\alpha)$ for some $\alpha \in \Delta$.

  If $us_\alpha \notin \idweak{w,\demazure{\Q}}$ then
  \[
    u(\alpha) \in \Eplus{w}{\demazure{\Q}} \subseteq \cone(\Roots{I}).
  \]
  Otherwise $\kappa(u) \neq \kappa(us_\alpha)$ and therefore by \Cref{lem:weakcoversunderkappa} we have $u(\alpha)$ is the direction of the flip from $\kappa(u)$ to $\kappa(us_\alpha)$, hence $u(\alpha) \in \Roots{I}$.
  This concludes the statement.
\end{proof}

We can now proof \Cref{thm:brickpolynormalcone} and \Cref{thm:coarseningcoxeter}.

\begin{proof}[Proof of \Cref{thm:brickpolynormalcone}]
  Let $I \in \subwordComplex(\Q,w)$ be a facet with pointed root configuration.
  We write in the following $\normalcone{\mathcal{X}}$ for a pointed cone~$\mathcal X$ to denote the normal cone at its apex.
  \Cref{cor:vertexpointedcone} yields
  \[
    \normalcone{\brickVector{I}} = \normalcone{\cone\Roots{I}}.
  \]
  
  It is well-known that the (closure of the) normal cone of the fundamental chamber~$\mathcal C$ is generated by the simple roots $\normalcone{\mathcal{C}} = \cone(\Delta) = \cone(\Phiplus)$.
  \Cref{lem:rootconfunderkappa} then gives
  \[
   \normalcone{\cone\Roots{I}} = \bignormalcone{\bigcap\limits_{v \in \kappa^{-1}(I)} v( \normalcone{\mathcal{C}} )} = \bigcup\limits_{v \in \kappa^{-1}(I)} \bignormalcone{ v( \normalcone{\mathcal{C}} ) } = \bigcup\limits_{v \in \kappa^{-1}(I)} v( \mathcal{C} ),
  \]
  where the latter two equalities are elementary transformations.
\end{proof}

\begin{proof}[Proof of \Cref{thm:coarseningcoxeter}]
  It follows from \Cref{thm:brickpolynormalcone} that the maximal cones of $\normalfan{\brickPolytope(\Q,w)}$ are given as in the corollary.
\end{proof}

\subsection{Containment properties of brick polyhedra for a fixed word}
\label{sec:brickpolycontainment}

The considerations of brick polyhedra also for non-spherical subword complexes makes it possible to discuss how the various brick polyhedra for a fixed word~$\Q$ are related.
In this section, we use the map $\iota : \subwordComplex(\Q,s_\beta w) \rightarrow \subwordComplex(\Q, w)$ from~\eqref{eq:iota}, to show how brick vectors change between the subword complexes $\subwordComplex(\Q,w)$ and $\subwordComplex(\Q,s_\beta w)$ for a cover $w \prec s_\beta w \leq \demazure{\Q}$ in Bruhat order.
Using \Cref{cor:weakordercplus}, we then obtain the following containment statement of brick polyhedra.

\begin{theorem}
  \label{thm:brickpolycontainment}
  Let $w \in W$ and $s \in \sref$ such that $w \prec ws \leq \demazure{\Q}$.
  Then $\brickPolytope(\Q,ws) \subseteq \brickPolytope(\Q,w)$.
\end{theorem}

We first show the following result for brick vectors and extract the key ingredient into a lemma.

\begin{proposition}
\label{prop:brickrelation}
  Let $w \prec s_\beta w \leq \demazure{\Q}$ and $I \in \subwordComplex(\Q, s_\beta w)$ be a facet.
  Then $\brickVector{I} \in \brickPolytope(\Q,w)$.
\end{proposition}

\begin{lemma}
  In the situation of the proposition, we have $\brickVector{I} \in \brickVector{ \iota(I) } + \RR_+(\beta)$.
\end{lemma}

\begin{proof}
  The facet $\iota(I)$ is the set $I \cup \{j\}$ where~$j$ is the unique position in the complement of~$I$, such that $\Root{I}{j} = \beta$ in $\subwordComplex(\Q, s_\beta w)$.
  Similar to \Cref{lem:weightssandflips}\ref{eq:weightsbyflips2}, the weight function $\Weight{\iota(I)}{\cdot}$ is obtained from the weight function $\Weight{I}{\cdot}$ by
  \[
    \Weight{\iota(I)}{k} = \begin{cases}
                                 \Weight{I}{k}            & \text{ if } k \leq j, \\
                                 s_\beta( \Weight{I}{k} ) & \text{ if } k > j.
                               \end{cases}
  \]
  Let $k > j$.
  Since $\beta = \Root{I}{j}$ and $j \notin I$, \Cref{lem:weightssandflips}\ref{eq:weightsbyflips3} yields that $\langle \beta, \Weight{I}{k} \rangle$ is non-positive.
  Therefore, $s_\beta( \Weight{I}{k} ) \in \Weight{I}{k} + \RRpos(\beta)$, and we obtain
  $\brickVector{I} = \brickVector{\iota(I)} + \RRpos(\beta)$.
\end{proof}

\begin{proof}[Proof of \Cref{prop:brickrelation}]
  The previous lemma shows that $\brickVector{I} \in \brickVector{ \iota(I) } + \RR_+(\beta)$.
  Since $\beta \in \Cplus{w}{\demazure{\Q}}$, the statement follows.
\end{proof}

\begin{proof}[Proof of \Cref{thm:brickpolycontainment}]
  This follows from \Cref{prop:brickrelation} together with the containment of Bruhat cones described in \Cref{cor:weakordercplus}.
\end{proof}

The following example shows the nested situation of brick polyhedra for the permutahedron in type~$A_2$.

\begin{example}[Type $A_2$]
  Let $\Q = 112211$.
  We show all brick polyhedra where brick vectors of different polyhedra that have the same coordinates are drawn close to each other.
  The brick polyhedra $\subwordComplex(\Q,w)$ are in black for $w = 121$, in red for $w = 12$, in blue for $w = 21$, in orange for $w = 1$, in lightblue for $w = 2$ and in grey for $w = e$.
  \[
    \begin{tikzpicture}[scale=1]
      \node (121a) at (-2.598,1.5) {$\bullet$};
      \node (121b) at (-1.299,2.25) {$\bullet$};
      \node (121c) at (-1.299,3.75) {$\bullet$};
      \node (121d) at (-2.598,4.5) {$\bullet$};
      \node (121e) at (-3.897,3.75) {$\bullet$};
      \node (121f) at (-3.897,2.25) {$\bullet$};
      \node (121g) at (-2.598,3) {$\bullet$};
      
      \draw (-2.598,1.5) -- (-1.299,2.25) -- (-1.299,3.75) -- (-2.598,4.5) -- (-3.897,3.75) -- (-3.897,2.25) -- (-2.598,1.5);
      \draw[fill=black, opacity=0.05] (-2.598,1.5) -- (-1.299,2.25) -- (-1.299,3.75) -- (-2.598,4.5) -- (-3.897,3.75) -- (-3.897,2.25) -- (-2.598,1.5);
      
      \node[red] (12a) at (-2.598,1.35) {$\bullet$};
      \node[red] (12b) at (-3.997,3.85) {$\bullet$};
      \node[red] (12c) at (-3.997,2.15) {$\bullet$};
      \node[red] (12d) at (-2.598,2.85) {$\bullet$};
      
      \draw[red] (-3.997,3.85) -- (-3.997,2.15) -- (-2.598,1.35);
      \draw[red, -{>[scale=1.2]}] (-3.997,3.85) -- (0, 6.1);
      \draw[red, -{>[scale=1.2]}] (-2.598,1.35) -- (3.897, 5.1);
      \draw[fill=red, opacity=0.05] (3.897, 5.1) -- (-2.598,1.35) -- (-3.997,2.15) -- (-3.997,3.85) -- (0.15, 6.2) -- (3.897,6.2);
      
      \node[blue] (21a) at (0,-0.25) {$\bullet$};
      \node[blue] (21b) at (1.299,0.6) {$\bullet$};
      \node[blue] (21c) at (0,1.5) {$\bullet$};
      \node[blue] (21d) at (1.299,2.35) {$\bullet$};
      
      \draw[blue] (0,-0.25) -- (1.299,0.6) -- (1.299,2.35);
      \draw[blue, -{>[scale=1.2]}] (1.299,2.35) -- (-5.196, 6.1);
      \draw[blue, -{>[scale=1.2]}] (0,-0.25) -- (-5.196, 2.75);
      \draw[fill=blue, opacity=0.05] (-5.196, 2.75) -- (0,-0.25) -- (1.299,0.6) -- (1.299,2.35) -- (-5.196, 6.1);
      
      \node[orange] (1a) at (0,-0.4) {$\bullet$};
      \node[orange] (1b) at (-1.299,0.35) {$\bullet$};
      \node[orange] (1c) at (-2.598,1.1) {$\bullet$};
      \node[orange] (1d) at (-4.1,1.98) {$\bullet$};
      
      \draw[orange] (0,-0.4) -- (-1.299,0.35) -- (-2.598,1.1) -- (-4.1,1.98);
      \draw[orange, -{>[scale=1.2]}] (0,-0.4) -- (3.897,2.15);
      \draw[orange, -{>[scale=1.2]}] (-4.1,1.98) -- (-4.1,6.1);
      \draw[fill=orange, opacity=0.05] (3.897,2.15) -- (0,-0.4) -- (-1.299,0.35) -- (-2.598,1.1) -- (-4.1,1.98) -- (-4.1,6.2) -- (3.897,6.2);
      
      \node[lightblue] (2a) at (0,-0.55) {$\bullet$};
      \node[lightblue] (2b) at (1.399,0.35) {$\bullet$};
      
      \draw[lightblue] (0,-0.55) -- (1.399,0.35);
      \draw[lightblue, -{>[scale=1.2]}] (1.399,0.35) -- (1.399,6.1);
      \draw[lightblue, -{>[scale=1.2]}] (0,-0.55) -- (-5.196,2.45);
      \draw[fill=lightblue, opacity=0.05] (-5.196,2.45) -- (0,-0.55) -- (1.399,0.35) -- (1.399,6.2) -- (-5.196, 6.2);
      
      \node[lightgrey] (ea) at (0,-0.7) {$\bullet$};
      
      \draw[lightgrey, -{>[scale=1.2]}] (0,-0.7) -- (-5.196,2.3);
      \draw[lightgrey, -{>[scale=1.2]}] (0,-0.7) -- (3.897,1.85);
      \draw[fill=lightgrey, opacity=0.1] (-5.196,2.3) -- (0,-0.7) -- (3.897,1.85) -- (3.897,6.2) -- (-5.196, 6.2);
      
      \draw (-5.196, 6.2) -- (3.897,6.2) -- (3.897,-0.8) -- (-5.196,-0.8) -- (-5.196, 6.2);
    \end{tikzpicture}
  \]
\end{example}

\bibliographystyle{alpha}
\bibliography{bpordsc}

\end{document}